\documentclass[11pt]{amsart}
\usepackage[utf8]{inputenc}
\usepackage{amssymb}
\usepackage{amsmath}
\usepackage{amsthm}
\usepackage{color}
\usepackage{ dsfont }
\usepackage[colorlinks=true, linkcolor=blue, citecolor=black]{hyperref}
\usepackage{graphicx}
\usepackage{multirow}
\usepackage{tikz}
\usepackage{enumitem}
\usepackage{comment}
\usepackage{bbm}

\newtheorem{maintheorem}{Theorem}

\newtheorem{theorem}{Theorem}[section]

\newtheorem{proposition}[theorem]{Proposition}

\newtheorem{lemma}[theorem]{Lemma}

\theoremstyle{remark}
\newtheorem{remark}[theorem]{Remark}

\theoremstyle{definition}
\newtheorem{definition}[theorem]{Definition}

\usepackage{array}
\usepackage{wrapfig}
\usepackage{multirow}
\usepackage{tabularx}
\usepackage{graphicx}
\usepackage{caption}
\usepackage{subcaption}
\usepackage{wrapfig}
\usepackage[skip=2pt,font=small]{caption}
{} 
{}   
{}   
{}    
\usepackage{authblk}

\newcommand{\N}{\mathbb{N}}
\newcommand{\Z}{\mathbb{Z}}
\newcommand{\R}{\mathbb{R}}
\newcommand{\T}{\mathbb{T}}
\newcommand{\Q}{\mathbb{Q}}

\newcommand{\X}{\mathcal{X}}
\newcommand{\Y}{\mathcal{Y}}

\newcommand{\U}{\mathcal{U}}
\newcommand{\e}{\varepsilon}
\newcommand{\E}{\mathcal{E}}

\newcommand{\norm}[1]{\left\lVert#1\right\rVert}

\newcommand{\M}{\mathcal{M}}

\renewcommand{\div}{\mathrm{div}\,}

\title[On some typicality and density results for vector fields]{On some typicality and density results \\ for nonsmooth vector fields \\ and the associated ODE and continuity equation}

\usepackage[foot]{amsaddr}
\author{Francesco Cianfrocca}
\author{Stefano Modena}
\address{GSSI, I-67100 L'Aquila, Italy}
\email{francesco.cianfrocca@gssi.it, stefano.modena@gssi.it}

\thanks{The authors wish to thank Luigi De Rosa for several intersting and useful discussion about the topic of this paper, in particular for pointing out Orlicz's result \cite{Orlicz32}. The authors thank the Max Planck Institute for Mathematics in the Sciences in Leipzig, where part of this work developed. Both authors are member of the GNAMPA group at INDAM}

\date{\today}
\begin{document}

\begin{abstract}
These notes address two problems. First, we investigate the question of ``how many'' are (in Baire sense) vector fields in $L^1_t L^q_x$, $q \in [1, \infty)$, for which existence and/or uniqueness of local, distributional solutions to the associated continuity equation holds. We show that, in certain regimes, existence of solutions (even locally in time, for at least one nonzero initial datum) is a meager property, whereas, on the contrary, uniqueness of solutions is a generic property. 
Secondly, despite the fact that non-uniqueness is a meager property, we prove that (Sobolev) counterexamples to uniqueness, both for the continuity equation and for the ODE, in the spirit of \cite{BrueColomboKumar2024} and \cite{kumar2023} respectively, form a dense subset of the natural ambient space they live in.

\end{abstract}

\maketitle

\section{Introduction}

In these notes we discuss some residuality and density properties for vector fields $u: [0,T] \times \T^d \to \R^d$ defined on the  $d$-dimensional flat torus $\T^d = \R^d/\Z^d$ and on some time interval $[0,T]$ with $T \in (0, + \infty)$. The paper is divided in two parts. In the first part we prove some residuality results for nonsmooth vector fields and the associated continuity equation, whereas in the second part we prove some density results concerning counterexample to uniqueness for the continuity equation and the ODE associated to nonsmooth vector fields.


\subsection{Residuality results for the continuity equation}
In the first part of the paper (Sections \ref{s:non-existence},\ref{s:residuality_uniqueness}) we consider vector fields
\begin{equation*}
    u \in L^1([0, T]; L^q(\T^d; \R^d)) =: L^1_t L^q_x
\end{equation*}
for some $q \in [1, \infty)$. For any given $u$, we consider the continuity equation 
\begin{equation}
\label{eq:PDE}
    \begin{cases}
    \partial_t \rho + \div (\rho u) & = 0, \\
    \rho|_{t=0} = \bar \rho
    \end{cases}   
\end{equation}
associated to $u$, where 
\begin{equation*}
    \bar\rho \in L^{q'}(\T^d)
\end{equation*}
is a given initial datum and we look for (local) distributional solutions to \eqref{eq:PDE}, i.e. maps $$\rho\in L^{\infty}_{loc}([0,\tau);L^{q'}(\T^d))$$ defined on some time interval $[0,\tau)$, $\tau \in (0,T]$, solving the Cauchy problem \eqref{eq:PDE} in the sense of distributions on $[0,\tau) \times \T^d$, i.e.  
    \begin{equation}
    \label{eq:def-weak-sln}
        \int_{0}^{\tau}\int_{\T^d} (\partial_t \varphi(t,x) + u(t,x)\cdot \nabla \varphi(t,x)) \rho(t,x) \,dx\,dt +\int \varphi(0,x) \bar\rho(x) dx =0
    \end{equation}
for all $\varphi\in C^{1}_c([0,\tau)\times \T^d).$ The choice of considering local solution is dictated by the fact that, without additional assumptions on $\div u$, solutions may not be defined on the whole time interval $[0, T]$. Notice also that considering initial data in $L^{q'}$ and looking for solutions in $L^\infty_{t,loc} L^{q'}_x$ is a natural choice, because in this way, $\rho u$ is locally integrable on $[0,\tau) \times \T^d$ and thus both $\rho$ and $\rho u$ have a meaning as distributions.

Goal of the first part of the paper is to investigate the following general question (stated for the moment in a quite rough way): how many are (in terms of Baire category arguments\footnote{Given a complete metric space $(X,d)$ and a subset $E \subseteq X$, we say that
\begin{enumerate}
    \item $E$ is \emph{nowhere dense} if its closure has empty interior;
    \item $E$ is of \emph{first category} or \emph{meager} if it can be written as a countable union of nowhere dense sets;
    \item $E$ is \emph{residual} if its complement $E^c = X \setminus E$ is of first category.
\end{enumerate}
Roughly speaking, first category and residuality are a ``topological version'' of, respectively, zero measure and full measure subsets of a measure space. As it happens for measure spaces, a countable union of first category sets is still first category and countable intersection of residual sets is still residual. Similarly, if $E$ is of first category and $A \subseteq E$, then $A$ is first category as well; if $E$ is residual and $E \subseteq B$, then $B$ is residual.}) the vector fields in $L^1_t L^q_x$, $q \in [1, \infty)$ for which existence and/or uniqueness of (local) solutions to \eqref{eq:PDE} holds?

\subsubsection{The case of continuous vector fields}
\label{sss:intro-continuous-vf}

The motivation for asking this general question comes from the fact that, in the case one considers continuous (in space) vector fields $u \in L^1([0,T]; C(\T^d; \R^d))$ and \emph{measure valued solutions} to the continuity equation
\begin{equation}
    \label{eq:measure-valued}
    \left\{
    \begin{aligned}
        \partial_t \mu_t + \div (u \mu_t) & = 0, \\
        \mu_0 & = \bar \mu  \in \mathcal{M}(\T^d) := \{\text{Radon measures on } \T^d\},
    \end{aligned}
    \right.
\end{equation}
a definite answer to this question is available.

First of all, recall that a \emph{measure valued weak solution} to \eqref{eq:measure-valued} is a family of Radon measures $\{ \mu_t\}_{t \in [0,T]}$ with ${\rm ess \, sup}_{t \in [0,T]} \| \mu_t\|_{\M(\T^d)} < \infty$ (here $\|\cdot\|_{\M(\T^d)}$ denotes the standard total variation norm on the space $\M(\T^d)$ of Radon measures) 
such that \eqref{eq:def-weak-sln} holds with $\rho(t,x)dx$ and $\bar \rho(x)dx$ substituted respectively by $d \mu_t(x)$ and $d \bar \mu(x)$. Notice that the space of Radon measures is, in the case of continuous vector fields, by duality, the natural space for the densities to live in. 

Concerning the \emph{existence property}, it is well known that \emph{for every}  initial datum $\bar \mu \in \mathcal{M}(\T^d)$, there is at least one (measure valued) weak solution $\{\mu_t\}_{t \in [0, T]} \subseteq \mathcal{M}(\T^d)$

The proof follows from a compactness argument: one considers a regularization $u^\e$ of the vector field $u$, with $u^\e \to u$ strongly in $L^1([0,T]; C(\T^d))$, and the associated continuity equation \eqref{eq:measure-valued} with $u^\e$ instead of $u$. For each $\e>0$, the (global) solution $\{\mu^\e_t\}_{t \in [0,T]}$ to the regularized equation is defined by
\begin{equation}
\label{eq:meas-sln}
\mu^\e_t := X^\e(t)_\sharp \bar \mu,
\end{equation}
where $X^\e$ is the flow map associated to $u_\e$, i.e. the solution to 
\begin{equation*}
    \partial_t X^\e (t,x) = u^\e (t, X^\e(t,x)), \qquad X^\e(0,x) = x.
\end{equation*}
From \eqref{eq:meas-sln} it also follows that
\begin{equation}
    \label{eq:bb-div-meas}
    \sup_{t \in [0,T]} \| \mu^\e_t\|_{\M(\T^d)}  = \|\bar \mu\|_{\M(\T^d)}.
\end{equation}
 
From \eqref{eq:bb-div-meas}, a weakly* convergining subsequence $\mu^{\e_n}_t \overset{*}{\rightharpoonup}  \mu_t$ (for a.e. $t$) is found. Finally the linearity of the equation implies that $\mu$ is a weak solution to \eqref{eq:PDE}.

On the other side, concerning the \emph{uniqueness property}, the Cauchy-Lipschitz theory ensures that, if $u \in L^1_t C_x$ is uniformly Lipschitz continuous (in the space variable), then uniqueness of solutions to \eqref{eq:measure-valued} holds. In terms of Baire category argument, this only implies that there is a meager set (namely the set of Lipschitz vector fields) for which uniqueness of solutions to \eqref{eq:measure-valued} holds, but, in principle, vector fields for which uniqueness holds may be many more. Indeed, a result by Bernard \cite{Bernard2008} shows that the set of all vector fields $u \in L^1_t C_x$ for which uniqueness of solutions to \eqref{eq:measure-valued} holds is residual (in the sense of Baire). 

Summarizing, we can say that for almost all continuous vector field (in Baire sense), the continuity equation is well posed in the space of measure valued solutions.  

For rougher vector field (precisely for vector fields in $L^1_t L^q_x$, $q< \infty$, which is the case considered in this paper) it turns out that the situation is quite different as far as the existence property is considered, whereas it is very similar as far as uniqueness property is considered. 

\subsubsection{Existence of local solutions for $u \in L^1_t L^q_x$}

As for continuous vector fields, also in the case of $u \in L^1_t L^q_x$ existence of solutions to the Cauchy problem \eqref{eq:PDE}  is usually proven by a compactness argument, where the counterpart to crucial estimate \eqref{eq:bb-div-meas} is the estimate

\begin{equation}
\label{eq:bb-div}
    \| \rho_\e\|_{L^\infty ([0,T]; L^{q'}(\T^d))} \leq M \| \bar \rho\|_{L^{q'}},
\end{equation}
for some $M>0$, where 
\begin{equation}
    \label{eq:rho-sln}
    \rho_\e(t, \cdot) : = X^\e(t)_\sharp \bar \rho,
\end{equation}
 similar to \eqref{eq:meas-sln}.

Differently than in the case of measure valued solutions, the crucial bound \eqref{eq:bb-div} can not be obtained for free from \eqref{eq:rho-sln} (recall that $q<\infty$ and thus $q'>1$). In general, additional properties on $u$ are assumed, related to its rate of ``compressibility''.

Standard situations are the following two. If $u$ is incompressible, then \eqref{eq:bb-div} holds as an equality with $M=1$. More in general, if $\div u \in L^\infty([0,T] \times \T^d)$, then \eqref{eq:bb-div} holds with $M$ depending on $\|\div u\|_{L^\infty}$. Slightly less standard, but still very well known, is the case of \emph{nearly incompressible} vector fields \cite{Bressan2003,NoteCamillo2007}, i.e. vector fields $u \in L^1_t L^q_x$ for which there is a weak solution $\vartheta \in L^\infty ([0,T] \times \T^d)$ to the continuity equation $\partial_t \vartheta + \div (\vartheta u) = 0$ and a constant $C>0$ such that $1/C \leq \vartheta \leq C$. Also in this case \eqref{eq:bb-div} holds, with $M$ depending on the constant $C$.

In all cases mentioned above, existence of \emph{global} solutions to \eqref{eq:PDE} \emph{for all} initial data $\bar \rho \in L^{q'}(\T^d)$ holds. To the best of our knowledge, no other sufficient or necessary criteria on the vector field $u \in L^1_t L^q_x$ are known which ensure the existence of (at least local) solutions for (at least some) initial data $\bar \rho \in L^{q'}(\T^d)$. 

Nevertheless, one would naively expect that, for a generic vector field $u \in L^1_t L^q_x$, existence of local solutions to \eqref{eq:PDE} holds, if not \emph{for all} initial data, at least for some nonzero initial datum. 

For instance, it is not difficult to see that for the (autonomous and bounded) vector field $u(x) = - x/|x|$ existence of \emph{local} solutions holds for all initial data which are supported far away from the origin\footnote{This is not a periodic vector field, but it is possible to design periodic analogs of this example.}. This vector field will also turn out to be the building block of the proof of Theorem \ref{thm:existence-lq} in Section \ref{s:non-existence}.

We thus define, for any given $u \in L^1_t L^q_x$, the set
    \begin{equation}
    \label{eq:exist-set}
    \begin{aligned}
            \mathcal{D}_q(u) & := \big\{ \bar \rho \in L^{q'} \,:\, \text{$\exists \, \tau>0 \text{ and }\rho\in L^{\infty}_{loc}([0,\tau);L^{q'}(\T^d))$ local weak solution} \\ 
            & \qquad  \text{to \eqref{eq:PDE}} \text{ with } \rho|_{t=0} = \bar \rho \big\}
    \end{aligned}
    \end{equation}
of all initial data admitting at least one local solution. Clearly it always holds $0 \in \mathcal{D}_q(u)$. Therefore the mildest \emph{existence property} we can require on a given vector field $u$ is that $\mathcal{D}_q(u)\supsetneqq \{0\}$, i.e. requiring that there is at least one (nonzero) initial datum from which it stars at least one solution, defined at least on a local time interval. 

For nearly incompressible vector fields (and thus also for divergence free or bounded divergence vector fields), $\mathcal{D}_q(u) = L^{q'}(\T^d)$. Similarly, the naive expectation discussed above can be phrased by saying that, for a generic $u \in L^1_t L^q_x$, the very mild existence property $\mathcal{D}_q(u) \supsetneqq \{0\}$ holds. 

The following theorem is the first result of the paper and it states such naive expectation is completely wrong:  almost no (in Baire sense) vector field $u \in L^1_t L^q_x$ has such existence property. 
\begin{maintheorem}[Residuality of non-existence]
\label{thm:existence-lq}
    Let $q \in [1, \infty)$. The set
    \begin{equation*}
       \E_q:= \{ u \in L^1([0,T]; L^q(\T^d)) \, : \, \mathcal{D}_q(u) \supsetneqq \{0\} \}
    \end{equation*}
    is meager in $L^1([0,T]; L^q(\T^d))$. 
\end{maintheorem}

\subsubsection{Uniqueness of solutions}

Concerning the uniqueness of solutions to \eqref{eq:PDE}, for vector fields in  $L^1_t L^q_x$, the situation is very similar to the case of continuous vector fields discussed above. Indeed, if $u \in L^1_t L^q_x$ is Lipschitz continuous (in the space variable), then uniqueness of solutions to \eqref{eq:PDE} in $L^\infty_t L^{q'}_x$ holds. 

Nevertheless, we prove below that, similarly to Bernard's result \cite{Bernard2008} for continuous vector fields, many more vector fields in $L^1_t L^q_x$ than Lipschitz continuous ones have the uniqueness property.   In view of the fact that existence of solutions is not always guaranteed, we first give the following definition.

\begin{definition}
    A vector field $u\in L^{1}
([0,T];L^q(\T^d;\R^d))$ has the \emph{uniqueness property} if the following holds. Let $\bar\rho \in \mathcal{D}_q(u)$. If  $\rho_1,\rho_2\in L^{\infty}_{loc}([0,\tau);L^{q'}(\T^d))$ are two local solutions (defined on the same time interval $[0,\tau)$) with the same initial datum $\bar\rho$, then $\rho_1 \equiv \rho_2$. We also set
    \begin{equation}
    \label{eq:uniq-set}
        \mathcal U_q := \big\{ u \in L^1_t L^q_x \, : \,
        \text{$u$ has the uniqueness property} \big\}.
    \end{equation}
\end{definition}

We now state the second main result of this paper, and then we comment on it. 

\begin{maintheorem}[Residuality of uniqueness]
 \label{thm:uniq-resid}
    Let $\mathcal X \subseteq L^1([0,T]; L^q(\T^d))$ and let $d_\X$ be a distance on $\X$ such that 
    \begin{enumerate}[label=(\roman*)]
        \item     $(\X, d_\X)$ is a complete metric space,
        \item\label{it:distances}  $\|u-v\|_{L^1_t L^q_x} \leq d_\X(u,v)$.
    \end{enumerate}
    If $\mathcal{U}_q \cap \X$ is dense in $(\X, d_\X)$, then it is residual in $(\X, d_\X)$.
\end{maintheorem}

We state the residuality in an arbitrary subset $\X$ of $L^1_t L^q_x$ (endowed with a distance which makes $\X$ a complete metric space), instead of looking just at the residuality of $\U^q$ in the whole $L^1_t L^q_X$ because, in this way, the statement applies to several different situations. 

For instance, if we take $\X = L^1_t L^q_x$ (the whole space), with $d_\X$ the $L^1_t L^q_x$ distance, then we get the counterpart of Bernard's result \cite{Bernard2008} for $L^1_t L^q_x$ vector fields, namely that $\U^q$ is residual in $L^1_t L^q_x$. 

Similarly, we can consider  
(for fixed $C>0$)
\begin{equation*}
    \begin{aligned}
        \X & = \{ u \in L^1_t L^q_x \, : \, \div u = 0 \}, \text{ or} \\
        \X & = \{ u \in L^1_t L^q_x \, : \, \|\div u\|_{L^\infty_{tx}} \leq C \},\text{ or} \\
        \X & = \{ u \in L^1_t L^q_x \, : \, \text{$u$ is nearly incompressible} \\
        & \qquad \qquad \qquad \quad \text{and there is $\vartheta$ solving \eqref{eq:PDE} with $1/C \leq \vartheta \leq C$} \},
    \end{aligned}
\end{equation*}
again with $d_\X$ being the $L^1_t L^q_x$ distance. In all three cases:
\begin{itemize}
    \item for all $u \in \X$, $\mathcal{D}_q(u) = L^{q'}(\T^d)$, i.e. existence of global solutions is always guaranteed. It thus makes very much sense to investigate if and how often also the uniqueness property holds. 
    \item  $\X$ is a meager subset of $L^1_t L^q_x$, hence the fact that $\U^q$ is residual in $L^1_t L^q_x$ gives us no information about the residuality of $\X \cap \U^q$ in $\X$. 
    \item  Nevertheless, $\X$ is a closed subset of $L^1_t L^q_x$, hence $(\X, d_\X)$ is a complete metric space, if $d_\X$ is the distance induced by the $L^1_t L^q_x$ topology. Now, smooth functions are dense in $\X$, hence, in particular, also $\U_q \cap \X$ is dense in $X$. Therefore  Theorem \ref{thm:uniq-resid} implies that $\U_q \cap \X$ is also residual in $\X$. 
\end{itemize}
Summarizing, for almost all (in Baire sense) vector fields in $\X$ (in all three cases above), there exists a unique solution $\rho \in L^\infty_t L^{q'}_x$ to \eqref{eq:PDE}. 

In view of the third main result of these notes, namely Theorem \ref{thm:density},  another possible application of Theorem \ref{thm:uniq-resid} is the following. Let 
\begin{equation*}
    \X := L^1_t L^q_x \cap L^1_t W^{1,p}_x \cap \{\div u = 0\},
\end{equation*}
(for $p \in [1, q)$), endowed with the distance induced by the norm $\|u\|_{L^1_t (L^q \cap W^{1,p})_x} := \|u\|_{L^1_t L^q_x} + \|u\|_{L^1_t W^{1,p}_x}$\footnote{We do not consider the case $p \geq q$, because in that case, by DiPerna and Lions theory, we would simply have $\U^q \cap \X = \X$.}. Clearly (i) and (ii) in the statement of Theorem \ref{thm:uniq-resid} holds. Moreover, smooth maps (and thus also $\U_q \cap \X$) are dense in $\X$. Therefore Theorem \ref{thm:uniq-resid} implies that for almost all (in Baire sense, w.r.t. the strong norm $\|\cdot\|_{L^1_t (L^q \cap W^{1,p})_x}$ ) vector fields in $\X = L^1_t(L^q \cap W^{1,p})_x \cap \{\div u = 0\}$, there exists a unique solution $\rho \in L^\infty_t L^{q'}_x$ to \eqref{eq:PDE}.

\subsubsection{Comments on literature}

We are not aware of other results where the question about  ``how many'' are (in terms of Baire category) those vector fields for which existence of local solutions to the PDE \eqref{eq:PDE} (for at least some initial data) holds.

On the contrary, concerning uniqueness this is not the first work where such question has been raised.

There are essentially two groups of results in the literature. 

The first group of results claims, in the same spirit as our  Theorems \ref{thm:uniq-resid}  that, roughly speaking, vector fields for which uniqueness holds are residual. We have already mentioned Bernard's work \cite{Bernard2008} for continuous vector fields and measure valued solutions to \eqref{eq:measure-valued}. A similar result, but for ODE, was proven by Orlicz \cite{Orlicz32} in 1932. Our Theorem \ref{thm:uniq-resid} was indeed inspired by \cite{Orlicz32}.

Lions proved in \cite{Lions98} that vector fields for which there exists a unique regular Lagrangian flow are residual in $L^1 ([0,T]; L^1(\T^d))$. 

The notion of regular Lagrangian flow is stricly linked to the one of \emph{renormalized solution} to the continuity equation \eqref{eq:PDE}: a distributional solution $\rho$ to \eqref{eq:PDE} is said \emph{renormalized} if $\beta(\rho)$ is also a distributional solution for all $\beta \in C^1(\R;\R)$ bounded and with bounded derivative. Lions' result in particular implies the residuality in $L^1_t L^1_x$ (and also in $L^1_t L^q_x$) of vector fields for which there exists a unique \emph{renormalized} solutions. The set of such vector fields is, in general, much larger than the set in \eqref{eq:uniq-set}, because every renormalized solution is (by definition) a weak solution, but the viceversa is not true in general (see e.g. \cite{BrueColomboKumar2024}). Hence our Theorem \ref{thm:uniq-resid} is stronger than what can be deduced by Lions' residuality result on regular Lagrangian flows.

The second group of results about ``how many'' are vector fields for which uniqueness of solutions to the PDE \eqref{eq:PDE} holds claims, roughly speaking, exactly the opposite of what we discussed so far. Namely, they claim that the vector fields for which non-uniqueness holds are residual, or, in other words, that a ``generic'' vector fields can not have the ``uniqueness property''. Usually such results are proven by convex integration methods.

In \cite{Crippa2015} Crippa, Gusev, Spirito and Wiedemann consider the set 
\begin{equation*}
    \X = \big\{ u : [0,1] \times \T^d \to \R^d \, : \, |u| \leq 1 \text{ almost everywhere} \big\},
\end{equation*}
endowed with the \emph{weak}-$L^2$ topology. They show that vector fields in $\X$ for which \emph{non-uniqueness} of bounded distributional solutions to the continuity equation \eqref{eq:PDE} holds are \emph{residual} in $X$. 

Similarly, in \cite{SattigLaszlo23} Sattig and Szekelyhidi consider the set 
\begin{equation*}
    \X = \big\{ u  \in C([0,1]; L^q \cap W^{1,p}(\T^d; \R^d))  \, : \, \|u\|_{L^q} \leq C \big\},
\end{equation*}
(with $p,q$ such that $W^{1,p} \not \hookrightarrow L^q$), endowed with the $C_t W^{1,p}_x$ strong topology. Notice that the uniform bound on $L^q$ ensures that $(\X, \| \cdot \|_{C_t W^{1,p}})$ is a complete metric space.

As before, the authors of \cite{SattigLaszlo23} show that vector fields in $\X$ for which \emph{non-uniqueness} of distributional solutions $\rho \in C_t  L^{q'}_x$ to the continuity equation \eqref{eq:PDE} holds are \emph{residual} in $\X$.\footnote{To be precise, the statement of the main theorem in \cite{SattigLaszlo23} is different, as it considers pairs $(u,\rho)$ of (sub)solutions to \eqref{eq:PDE}, but we preferred to mention it in the (slightly imprecise) form above to make it consistent with the rest of the discussion in this introduction.}

There is an apparent contradiction between the first group of results we mentioned (including our Theorem \ref{thm:uniq-resid}), claiming that ``good'' vector fields (i.e. the ones for which uniqueness holds) are residual and the second group of results, claiming that ``bad'' vector fields are residual (and thus  ``good'' ones are of first category). In our opinion, this apparent contradiciton is due to the topology that is considered. In all results where ``good fields are residual'' the space of vector fields is endowed with the (natural) strong topology. In all results where ``bad fields are residual'' the space of vector fields  is endowed with a weaker topology, namely the weak topology in the case of \cite{Crippa2015} and the $W^{1,p}$ topology (which is weaker than the natural $L^q$ topology since, by assumption, $W^{1,p} \not \hookrightarrow L^q$) in \cite{SattigLaszlo23}.

In a word, the situation seems to be the following: when the natural strong topology is considered, almost all vector fields (in the sense of Baire) have the ``uniqueness property''; on the contrary, when a weaker topology is considered, almost all vector fields (again in Baire sense) present non-uniqueness issues. 

We conclude by remarking that the goal of \cite{Crippa2015} and \cite{SattigLaszlo23} was to show the existence of infinitely many vector fields with non-uniqueness, hence it is not really important what topology is put on the space of vector fields under consideration. On the other side, as we pointed out at the beginning of this introduction, our goal in these notes (as well as in the mentioned results by Bernard \cite{Bernard2008} and Lions \cite{Lions98}) is  to understand ``how many'' are vector fields for which uniqueness/non-uniqueness hold, and thus, in this sense, the strong topology (on the space of vector fields under consideration) seems to be the natural choice.

\subsection{Density results for counterexample to uniqueness}

The second part of these notes (Sections \ref{s:density-pde}, \ref{s:density-ode}) concerns a related, but rather different question. In the recent years there has been a considerable interest in the construction of vector fields $u$ having the highest possible regularity (both in terms of incompressibility property and Sobolev regularity) for which uniqueness of solutions to the continuity equation \eqref{eq:PDE} or to the ODE 
\begin{equation}
    \label{eq:ODE1}
    \gamma'(t) = u(t,\gamma(t)), \qquad \gamma(0) =x
\end{equation}
fails. 

More precisely, concerning the PDE \eqref{eq:PDE}, the theory of DiPerna and Lions \cite{DiPernaLions89} ensures that if $u \in L^1_t W^{1,p}_x$ and it has bounded divergence, $\bar \rho \in L^{q'}_x$ and 
\begin{equation}
    \label{eq:dpl}
    \frac{1}{p} \leq  \frac{1}{q},
\end{equation}
(or, as it is commonly written  $1/p + 1/q' \leq 1$) then there exists a unique weak solution $\rho \in L^\infty_t L^{q'}_x$. A relevant question is thus whether condition \eqref{eq:dpl} is sharp. First answers to these questions were given using convex integration in \cite{Modena2018}, \cite{Modena2019}, \cite{Modena2020}, where it was shown that uniqueness may fail if
\begin{equation}
    \label{eq:ms-range}
        \frac{1}{p} > \frac{1}{q} + \frac{1}{d}.
\end{equation}

More recently Bru\`e, Colombo and Kumar \cite{BrueColomboKumar2024} showed that uniqueness may fail also in the larger (and to some extent optimal) range
    \begin{equation}
    \label{eq:bck-range}
        \frac{1}{p} > \frac{1}{q} + \frac{1}{d-1}\frac{p-1}{p}.
    \end{equation}
The construction in \cite{BrueColomboKumar2024} is not based on convex integration, but it is rather a quite explicit construction (performed by a fixed point method) of a particular vector field and a particular non-renormalized solution to \eqref{eq:PDE}. 

Similarly, concerning the ODE \eqref{eq:ODE1} for Sobolev vector fields, it was shown by Caravenna and Crippa in \cite{CaravennaCrippa16} that if $u \in L^1_t W^{1,p}_x$ with $p>d$ and $\div u \in L^\infty_{t,x}$, then
\begin{equation}
\label{eq:non-uniq-ode}
   A_u := \{ x_0 \in \T^d \, : \, \text{there are two different solutions to \eqref{eq:ODE1} starting from $x_0$} \}
\end{equation}
has zero Lebesgue measure. 
On the opposite side, Bru\`e, Colombo and De Lellis \cite{BrueColomboDelellis2021} constructed by convex integration vector fields  $u \in C_t W^{1,p}_x$, for any $p<d$, such that $A_u$ has positive measure (later improved in \cite{PitchoSorella23} where, by a similar construction, the authors were able to show that  $A_u$ may even have full measure). 

A big step forward came from the work of Kumar \cite{kumar2023}, where he constructed a vector field $u \in C_t W^{1,p}_x \cap C^\alpha_{t,x}$ (for any $p<d$ and $\alpha<1$, thus saturating both the Caravenna-Crippa range $p>d$ and the classical Cauchy-Lipschitz threshold) such that $A_u$ has full measure. As observed before for \cite{BrueColomboKumar2024} (which is actually based on \cite{kumar2023}), also this example  is not built by convex integration, but it is rather based on an explicit construction. 

Convex integration is usually refered to as a \emph{flexible} technique: an instance of such flexibility in the results mentioned above is that fact that, for instance, the construction in \cite{Modena2018, Modena2019, Modena2020} can be easily adapted to show that incompressible vector fields in $C_t (L^q \cap W^{1,p})$, with $q,p$ in the range \eqref{eq:ms-range}, for which uniqueness of solutions in $L^\infty_t L^{q'}_x$ fails are \emph{dense} in the ambient space of $C_t (L^q \cap W^{1,p})_x$ incompressible fields.  A similar density result can be stated for the counterexample to ODE uniqueness in \cite{BrueColomboDelellis2021}. 

On the other side, the situation might be different for the explicit counterexample to uniqueness provided in \cite{BrueColomboKumar2024} (for the PDE) and in \cite{kumar2023} (for the ODE). Such examples are more regular (in terms of Sobolev and/or H\"older regularity) than the examples constructed by convex integration, but they are based on a rather explicit construction of a specific vector field for which uniqueness fails, and thus it is not clear if the density property mentioned above continues to hold. 

Goal of the second part of the paper is precisely to understand whether such density property (and thus the \emph{flexibility} which comes \emph{for free} with convex integration) is shared also by the (more regular) explicit counterexamples to uniqueness constructed in \cite{BrueColomboKumar2024} and \cite{kumar2023} without convex integration. The answer is positive and it is stated in the following theorem.

Recall the definition of $\mathcal{U}_q \subseteq L^1_t L^q_x$ in \eqref{eq:uniq-set}. Recall also, for given $u \in C([0,T] \times \T^d)$, the definition \eqref{eq:non-uniq-ode} of the set $A_u$ of starting points for which uniqueness to \eqref{eq:ODE1} fails.

\begin{maintheorem}[Density of counterexamples]
Let $T>0$. 
The following holds.
    \label{thm:density}
    \begin{enumerate}[label=(\roman*)]
        \item  Let $q,p$ be as in \eqref{eq:bck-range}. Let 
        \begin{equation*}
            \Y := \{ u \in L^1([0,T];  (L^q \cap W^{1,p})(\T^d)) \, : \, \div u = 0\} 
        \end{equation*}
        endowed with the natural norm $\|u \|_{L^1_t(L^q \cap W^{1,p})_x}:=\|u\|_{L^1_t L^q_x} + \|u \|_{L^1_t W^{1,p}_x}$.
        The set $\U_q^c \cap \Y$ is dense in $\Y$. 
        \item  Let $1 \leq p<d$ and let 
        \begin{equation*}
            \mathcal{Z} := \{ u \in C([0,T] \times \T^d) \cap  L^1_t  W^{1,p}_x \, : \, \div u = 0\} 
        \end{equation*}
        endowed with the natural norm $\|u \|_{C_{t,x} \cap L^1_t W^{1,p}_x}:=\|u\|_{C_{t,x}} + \|u \|_{L^1_t W^{1,p}_x}$. For every $\alpha \in [0,1)$, the set 
        \begin{equation}
        \label{eq:statement-ode-nonuniq}
        \begin{aligned}
            \big\{ u \in \mathcal{Z}  \, : \, & u \in  C^\alpha_{t,x}  \text{ and $A_u$ has full measure} \}       
        \end{aligned}
        \end{equation}
        is dense in $\mathcal{Z}$.
    \end{enumerate}    
\end{maintheorem}

Some remarks about the statements. 

\begin{remark}
    In both cases the density holds in the closed subspace (of $L^1_t (L^q \cap W^{1,p})_x$ or $C_{t,x} \cap L^1_t W^{1,p}_x$ respectively) of divergence free vector field. This is natural, because we are investigating the density of the \emph{set of counterexamples} coming from \cite{BrueColomboKumar2024} and \cite{kumar2023}, which, by definition, contains only divergence free fields and therefore can not be dense in the whole ambient space.
\end{remark}

\begin{remark}
     In both cases (i) and (ii), the density statement can not come by Baire type arguments. Indeed:
        \begin{enumerate}
            \item concerning Point (i), by Theorem \ref{thm:uniq-resid}, $\U_q \cap \Y$ is residual in $\Y$ (endowed with the norm $\|u \|_{L^1_t(L^q \cap W^{1,p})_x}$ as in the statement of Theorem \ref{thm:density}). Hence $\U_q^c \cap \Y$ is meager in $\Y$
            and thus its density can not be deduced by Baire-type arguments.
            \item Concerning Point (ii), define
                    \begin{equation*}
                        \mathcal{O} := \big\{ u \in C([0,T]  \times \T^d) \, : \, \forall x_0 \in \T^d \  \exists! \, \gamma \text{ solving \eqref{eq:ODE1} in integral sense} \big\}
                    \end{equation*}
            A result by Orlicz \cite{Orlicz32} (for which Bernard's result \cite{Bernard2008} mentioned in Section \ref{sss:intro-continuous-vf}  and our Theorem \ref{thm:uniq-resid} are the PDE counterparts) states that $\mathcal{O}$ is residual in $C([0,T] \times \T^d)$. Orlicz's result can be easily adapted to show that $\mathcal{O} \cap L^1_t W^{1,p}_x \cap \{\div u = 0\}$ is as well residual in $C([0,T] \times \T^d) \cap L^1_t W^{1,p}_x \cap \{ \div u = 0\}$, endowed with the natural norm $\|u\|_{C_{t,x} \cap L^1_t W^{1,p}_x}$. Therefore the set in \eqref{eq:statement-ode-nonuniq} is meager in  $C([0,T] \times \T^d) \cap L^1_t W^{1,p}_x \cap \{ \div u = 0\}$, and thus, as for Point (i), its density can not be deduced by Baire-type arguments.
    \end{enumerate}
\end{remark}

\begin{remark}
    As immediate corollary of Theorem \ref{thm:density}, we observe that $\U_q^c \cap \Y$ is dense also in $L^1_t L^q_x \cap \{ \div u = 0\}$, endowed with the standard $L^1_t L^q_x$ topology. 
    Similarly, the set in \eqref{eq:statement-ode-nonuniq} is dense in $C([0,T] \times \T^d) \cap \{ \div u = 0\}$, endowed with the standard uniform topology. 
\end{remark}

\begin{remark}
    Concerning Point (ii) in the statement of Theorem \ref{thm:density}, it would be natural  to consider (for $\alpha \in (0,1)$ fixed) 
        \begin{equation*}
            \tilde{\mathcal{Z}} := \{ u \in C^\alpha([0,T] \times \T^d) \cap  L^1_t  W^{1,p}_x \, : \, \div u = 0\} 
        \end{equation*}
        endowed with the natural norm $\|u \|_{C^\alpha_{t,x} \cap L^1_t W^{1,p}_x}:=\|u\|_{C^\alpha_{t,x}} + \|u \|_{L^1_t W^{1,p}_x}$ and to ask whether the set
        \begin{equation*}
        \begin{aligned}
            \big\{ u \in \tilde{\mathcal{Z}}  \, : \, & u \in  C^\alpha_{t,x}  \text{ and $A_u$ has full measure} \}       
        \end{aligned}
        \end{equation*}
        is dense in $\tilde{\mathcal{Z}}$. Our proof however does not apply to this case, the main reason being that smooth maps are not dense in $C^\alpha_{t,x}$, see Remark \ref{rmk:holder-not-work}.
\end{remark}

The paper is organized as follows. In Section \ref{s:push-forward} a preliminary proposition concerning the composition of flow maps is proven. Such proposition will be used in both the proofs of Theorem \ref{thm:existence-lq} and Theorem \ref{thm:density}. Theorem \ref{thm:existence-lq} (residuality of non-existence) is proven in Section \ref{s:non-existence}. Theorem \ref{thm:uniq-resid} (residuality of uniqueness) is proven in Section \ref{s:residuality_uniqueness}. Theorem \ref{thm:density}, Point (i) (density of PDE counterexamples) is proven in Section \ref{s:density-pde}, whereas Theorem \ref{thm:density}, Point (ii) (density of ODE counterexamples) is proven in Section \ref{s:density-ode}.

\bigskip

\section{Composition of flow maps}
\label{s:push-forward}

In this section we observe that the composition (as functions of space) of two flow maps corresponding to two given vector fields $u,v$ is the flow map of a vector field $w$ which can be computed explicitly in terms of $u,v$. A similar result holds for solutions to the continuity equation. 

Consider a smooth vector field $u\in C^{\infty}([0,T]\times \T^d;\R^d)$, $T>0$. Denote by $X^u :[0,T]\times \T^d\to \T^d$ the associated flow map
\begin{equation}
\label{eq:flow-map}
    \left\{
    \begin{aligned}
    \partial_t X^u(t,x) & = u(t, X^u(t,x)), \\
    X^u(0,x) & = x,
    \end{aligned}
    \right.
\end{equation}
and by $\Phi^u : [0,T] \times \T^d \to \T^d$ the corresponding inverse flow map
\begin{equation*}
\Phi^u(t,X^u(t,x)) = x, \quad X^u(t, \Phi^u(t,y)) = y. 
\end{equation*}
Assume now that $v\in L^1([0,T];L^q(\T^d;\R^d))$, $q \in [1, \infty]$, is another vector field (not necessarily smooth). Define  
$w:[0,T]\times \T^d \to \R^d$ as
    \begin{equation}
    \label{eq:push_forward}
        w(t,x):=u(t,x)+\left(\nabla \Phi^u(t,x)\right)^{-1}v(t,\Phi^u(t,x)).
    \end{equation}
Observe that $w \in L^1_t L^q_x$. 
\begin{remark}
    \label{rmk:div-u-equal-div-w}
    If $u$ is incompressible, then $\Phi^u$ is a measure preserving diffeomorphism and thus it is not difficult to see (by testing with smooth functions and integrating by parts) that
    $$\div \Big( \left(\nabla \Phi^u(t,x)\right)^{-1}v(t,\Phi^u(t,x)) \Big) = (\div v )(t, \Phi^u(t,x)).$$ In particular, $\div w = 0$ if  $\div v = 0$.
\end{remark}

\begin{lemma}\label{lemma:composition_lemma}
    Let $u,v,w$ be as before. 
    The following hold.
    \begin{enumerate}[label=(\roman*)]
        \item Let $\rho\in L^\infty([0,T];L^{q'}(\T^d;\R))$ be a distributional solution to
        \begin{equation*}
            \left\{
            \begin{aligned}
                \partial_t \rho + \div (v \rho) & = 0, \\
                \rho|_{t=0} = \bar\rho
            \end{aligned}
            \right.
        \end{equation*}
        for some $\bar\rho\in L^{q'}(\T^d)$.
        Define $\tilde \rho\in L^\infty([0,T];L^{q'}(\T^d;\R))$ by
        \begin{equation*}
            \tilde \rho(t, \cdot)= X^u(t, \cdot)_\sharp \rho(t, \cdot) 
        \end{equation*}
        Then
         \begin{equation*}
            \left\{
            \begin{aligned}
                \partial_t \tilde \rho + \div (w \tilde  \rho) & = 0, \\
                \tilde \rho|_{t=0} = \bar\rho.
            \end{aligned}
            \right.
        \end{equation*}
        
    \item Assume now, in addition, that  $v\in C([0,T]\times \T^d)$. Let $\gamma^v:[0,T]\to\T^d$ be a solution to the ODE
    \begin{equation*}
        \left\{
        \begin{aligned}
            (\gamma^v)'(t) & = v(t, \gamma^v(t)), \\
            \gamma^v(0) & = x,
        \end{aligned}
        \right.
    \end{equation*}
    for some $x \in \T^d$. 
    Then $\gamma^w:[0,T]\to \T^d$ defined by 
    \begin{equation*}
        \gamma^w(t):=X^u(t,\gamma^v(t))
    \end{equation*}
    solves
    \begin{equation*}
        \left\{
        \begin{aligned}
            (\gamma^w)'(t) & = w(t, \gamma^w(t)), \\
            \gamma^w(0) & = x.
        \end{aligned}
        \right.
    \end{equation*}    
    \end{enumerate}
\end{lemma}

\begin{remark}
        In particular, if also $v \in C^\infty([0,T] \times \T^d)$ and $X^v:[0,T] \times \T^d \to \T^d$ is the corresponding flow map, then also $w \in C^\infty([0,T] \times \T^d)$ and its flow map $X^w$ satisfies
    \begin{equation*}
        X^w(t,x)=X^u(t,X^v(t,x)) \qquad\forall x\in\T^d,\,\forall t\in[0,T].
    \end{equation*}

\end{remark}


\begin{proof}
\underline{Proof of (i)}. We check that $\tilde \rho$ solves the Cauchy problem (with velocity field $w$) in the sense of distributions. Using the identity 
\begin{equation}
    \label{eq:inverse-matrix}
    (\nabla \Phi^u(t,X^u(t,x)))^{-1} = \nabla X^u(t,x)
\end{equation}
for all $t,x$, it is not difficult to see that 
\begin{equation}
\label{eq:w-new-coordinate}
    w(t, X^u(t,x)) = u(t, X^u(t,x)) + \nabla X^u(t,x) v(t,x)
\end{equation}
for a.e. $t,x$. Let now $\varphi \in C^\infty_c([0,T) \times \T^d)$ be a test function. Using the definition of $\tilde \rho$, we have
\begin{equation*}
    \begin{aligned}
        \iint \big( & \partial_t \varphi (t,x) + \nabla \varphi(t,x) \cdot w(t,x) \big) \tilde \rho(t,x) dxdt \\  
        & = \iint \big( \partial_t \varphi (t,x) + \nabla \varphi(t,x) \cdot w(t,x) \big)  X^u(t)_\sharp \rho(t) (dx) dt \\
        & = \iint \big( \partial_t \varphi (t,X^u(t,x)) + \nabla \varphi(t,X^u(t,x)) \cdot w(t,X^u(t,x)) \big) \rho(t,x) dx dt \\
        & \text{(by \eqref{eq:w-new-coordinate})} \\
        & = \iint \Big[ \underbrace{\partial_t \varphi(t,X^u(t,x)) + \nabla \varphi(t,X^u(t,x)) \cdot u(t,X^u(t,x))}_{= \partial_t \psi(t,x)} \\
        & \qquad \qquad\qquad \qquad \qquad\qquad +  \underbrace{\nabla \varphi(t,X^u(t,x)) \nabla X^u(t,x)}_{= \nabla \psi(t,x)} v(t,x) \Big] \rho(t,x) dxdt \\
        & \text{(setting $\psi(t,x) := \varphi(t,X^u(t,x))$)} \\
        & = \iint \big( \partial_t \psi(t,x) + \nabla \psi(t,x) \cdot v(t,x) \big) \rho(t,x) dxdt \\
        & = - \int \psi(0,x) \bar\rho(x) dx = - \int \varphi(0,x) \bar\rho(x) dx,
    \end{aligned}
\end{equation*}
where in the second last equality we used that $\rho$ is a distributional solution to  the continuity equation with vector field $v$ and $\psi \in C^\infty([0,T) \times \T^d)$ is an admissible test function (recall that $u$ is smooth and so is $X^u$), and in the last equality we used that $\psi(0,x) = \varphi(0, X^u(0,x)) = \varphi(0,x)$. 

    \medskip
    
    \underline{Proof of (ii).}
    For any time $t$ we can explicitly compute the ODE satisfied by $\gamma^w.$
    \begin{equation*}
        \begin{aligned}
            \frac{d}{dt}\gamma^w(t)&=\frac{d}{dt}X^u(t,\gamma^v(t))=\partial_tX^u(t,\gamma^v(t))+\nabla X^u(t,\gamma^v(t))\frac{d}{dt}\gamma^v(t)\\
            &=u(t,\gamma^w(t)) + \nabla X^u(t,\gamma^v(t))v(t,\gamma^v(t)).
        \end{aligned}
    \end{equation*}
Next we use the identity \eqref{eq:inverse-matrix} to obtain 
 \begin{equation*}
        \begin{aligned}
            \frac{d}{dt}\gamma^w(t)=u(t,\gamma^w(t)) + \left(\nabla \Phi^u(t,\gamma^w(t))\right)^{-1}v(t,\Phi^u(t,\gamma^w(t))) = w(t, \gamma^w(t)),
        \end{aligned}
    \end{equation*}
    which concludes the proof.
\end{proof}


\section{Residuality of non-existence}\label{s:non-existence}

In this section we prove Theorem \ref{thm:existence-lq}. We first state and prove in Section \ref{ss:ex:prelim} a preliminary standard lemma concerning the weak* modulus of continuity of distributional solutions to \eqref{eq:PDE}.  Then in Section \ref{ss:thm-ex} we state a \emph{Main Proposition}, namely Proposition \ref{prop:main_proposition_existence}, and we prove Theorem \ref{thm:existence-lq}, assuming the Main Proposition. Finally in Section \ref{sec:main_prop_existence} we prove the Main Proposition. 

\subsection{A preliminary lemma}
\label{ss:ex:prelim}

For every $u\in L^{1}([0,T];L^{q}(\T^d))$ and every local weak solution $\rho\in L^{\infty}_{\rm loc}([0,\tau);L^{q'}(\T^d))$ to \eqref{eq:PDE},  
it is a well known fact (see for example \cite[Theorem 4.3.1]{DafermosBook}) that, up to a modification on a negligible set of times, the density $\rho$ can be assumed 
weakly* continuous w.r.t time, i.e. 
\begin{equation*}
    \lim_{s\to t}\int_{\T^d}\rho(s,x)g(x)\,dx=\int_{\T^d}\rho(t,x)g(x)\,dx,\qquad \forall g \in L^{q}, \,\,\forall t\in[0,\tau).
\end{equation*}
We will then often implicitly assume that  $\rho$ is weakly* continuous in $L^{q'}$, and we will write $\rho\in C_{w}([0,\tau);L^{q'}(\T^d))$.

In fact, weakly* continuous solutions have a modulus of continuity in time w.r.t an appropriate distance. 
This simple and standard fact is going to be useful in the compactness argument in the proofs of Theorem \ref{thm:existence-lq} and \ref{thm:uniq-resid}. We include a proof for the sake of completeness.

\begin{lemma}\label{lemma:equicont_solutions}
Let $q \in [1, \infty)$ and $M > 0$.  Consider 
    \begin{equation*}
        \overline{B}_M^{q'} := \{f\in L^{q'}(\T^d)\,:\,\norm{f}_{L^{q'}(\T^d)}\le M\}
    \end{equation*}
    endowed with the weak* topology of $L^{q'}$. There exists a distance $d_M$ on $\overline{B}_M^{q'}$ with the following properties:
    \begin{enumerate}[label=(\roman*)]
        \item $d_M$ induces the weak* topology on $\overline{B}_M^{q'}$,
        \item  for every $u \in L^1([0,T]; L^q(\T^d;\R^d))$, for every $\tau \in (0,T]$ and every local solution $\rho \in C_w([0,\tau]; L^{q'}(\T^d))$ solution to \eqref{eq:PDE} such that $\rho(t)\in \overline{B}^{q'}_M$ for all $t\in [0,\tau]$, the map $t \mapsto \rho(t)$ is absolutely continuous w.r.t. $d_M$. More precisely
        \begin{equation*}
            d_M(\rho(t), \rho(s)) \leq M \int_{s}^t  \|u(z, \cdot)\|_{L^q} \, dz, \qquad \text{ for every $\tau\ge t\ge s\ge 0.$}
        \end{equation*}
    \end{enumerate}
\end{lemma}
\begin{proof}
    Let $ \{\varphi_j\}_{j\in\N}$ be a countable family of smooth maps, dense in $\bar{B}_1^{q}$ (w.r.t. the strong topology of $L^q$) with the additional property that 
    \begin{equation}
    \label{eq:bound-lipschitz}
        \| \nabla \varphi_j\|_{L^\infty} \geq 1. 
    \end{equation}
    The metric $d_M$ is defined in the standard way
	\begin{equation*}
        d_M(g,h):=\sum_{j\in\N}\frac{2^{-j}}{\norm{\nabla\varphi_j}_{L^{\infty}}}\left|\int_{\T^d}(g(x)-h(x)){\varphi}_j(x)\,dx\right|
	\end{equation*}
    for any $g,h\in \bar{B}_M^{q'}$, and it induces the weak* topology on $\bar B^{q'}_M$.  Notice that the bound \eqref{eq:bound-lipschitz} ensures that the series converges.
    Let now $u, \rho$ be as in Point (ii) of the statements. For $\tau\ge t\ge s\ge 0$, we have
    \begin{equation*}
        \begin{aligned}
            d_{M}(\rho(t),\rho(s))&=\sum_{j\in\N}\frac{2^{-j}}{\norm{\nabla{\varphi}_j}_{L^{\infty}}}\left|\int_{\T^d}(\rho(t,x)-\rho(s,x)){\varphi}_j(x)\,dx\right|\\
            &=\sum_{j\in\N}\frac{2^{-j}}{\norm{\nabla{\varphi}_j}_{L^{\infty}}}\left|\int_{s}^{t}\int_{\T^d}\rho(z,x)u(z,x)\cdot\nabla{\varphi}_j(x)\,dx\right|\\
            &\le \sum_{j\in\N} 2^{-j}M \int_{s}^t  \|u(z, \cdot)\|_{L^q} \, d z= M \int_{s}^t  \|u(z, \cdot)\|_{L^q} \, d z,
        \end{aligned}
    \end{equation*}
    where in the second line we have used the fact that $\rho$ is weakly* continuous in time and it is a distributional solution to \eqref{eq:PDE}. 
\end{proof}

\subsection{Main Proposition and proof of Theorem \ref{thm:existence-lq}}
\label{ss:thm-ex}

The proof of Theorem \ref{thm:existence-lq} consists on showing that the set $\E_q$ is contained into a countable union of closed sets with empty interior. 
The crucial step of the proof  is the one proving the statement about the empty interior of each of those sets.
This part will rely on the construction of smooth vector fields $w$,  arbitrarily close to a given field $u$, having the property of ``compressing the available mass'', thus leading to a sort of \emph{norm inflation} in solutions to the associated continuity equation, in an arbitrarily small amount of time. 
The construction can be summarized as follows.

\begin{proposition}[Main Proposition]\label{prop:main_proposition_existence}
    Let $u\in C^{\infty}([0,T]\times\T^d;\R^d)$. For every $\e>0$, $\delta>0$, $\tau>0$, $M>1$ and $\varphi\in C^\infty(\T^d)$, there exists a vector field $w\in C^{\infty}([0,T]\times\T^d;\R^d)$ such that 
    \begin{equation*}
        \norm{u-w}_{L^1_tL^q_x}\le  \epsilon
    \end{equation*}
    and having the following property.
    For every $\bar\rho \in L^{q'}(\T^d)$ such that
    $$\left| \int_{\T^d}\varphi(x)\bar\rho(x)\,dx \right|\ge \delta,$$
    the unique solution $\rho\in C_w([0,+\infty);L^{q'}(\T^d))$ to the Cauchy problem
    \begin{equation*}
       \left\{\begin{aligned}
            &\partial_t\rho + \div (\rho w)=0,\\
            &\rho_{t=0}=\bar\rho
       \end{aligned}\right.
    \end{equation*}
    satisfies
    \begin{equation}
    \label{eq:norm_inflation}
        \sup_{s\in[0,\tau]}\norm{\rho(s)}_{L^{q'}}> M.
    \end{equation}
\end{proposition}

\medskip

We postpone to Section \ref{sec:main_prop_existence} the proof of Proposition \ref{prop:main_proposition_existence} and we prove now Theorem \ref{thm:existence-lq}.

\begin{proof}[Proof of Theorem \ref{thm:existence-lq} assuming Proposition \ref{prop:main_proposition_existence}]

The proof is divided in three claims. First we write $\E_q$ as countable union of sets. Then we show that each of these sets is closed. Finally we show that each of these sets has empty interior. The first two claims are straightforward. The last claim follows form the main Proposition \ref{prop:main_proposition_existence}.

\medskip\textit{\underline{Claim 1}. Let $\mathcal{G}\subseteq C^{\infty}\cap L^{q}(\T^d)$ be a dense and countable set in $L^{q}(\T^d)$. The following inclusion holds: 
\begin{equation*}
    \E_q\subseteq \bigcup_{\tau\in \Q}\bigcup_{k= 1}^{\infty}\bigcup_{n=
    1}^{\infty}\bigcup_{\psi\in \mathcal{G}}\E_{\tau,k,n,\psi}
\end{equation*}
where} 
\begin{equation}\label{eq:def_iniseme_esistenza}
\begin{aligned}
    \E_{\tau,k,n,\psi}:= \big\{u\in &L^1_tL^q_x\,:\, \exists \bar \rho \in L^{q'} \text{ such that }\\
    &(i)\,\,\left|\int\psi(x)\bar\rho(x)\,dx\right|\ge \frac{1}{k},\\
    &(ii) \,\,\exists \rho\in C_{w}([0,\tau];L^{q'}(\T^d)),\,\,  \norm{\rho}_{L^{\infty}([0,\tau];L^{q'}(\T^d))}\le n,\,\,\\
     &\qquad \text{local solution to \eqref{eq:PDE} up to time $\tau$ with $\rho|_{t=0}=\bar\rho$}\big\} .
    \end{aligned}
\end{equation}

\smallskip

Let $u\in \mathcal{E}_q$. Then there exists $\bar\rho \in L^{q'}(\T^d)$, $\bar \rho \not \equiv 0$, a time $T_0\in(0,T]$ and a local solution $\rho\in L^{\infty}_{loc}([0,T_0);L^{q'}(\T^d))$.
First of all, for any $\tau <T_0$ (so in particular for $\tau$ rational), by definition $\rho\in L^{\infty}([0,\tau];L^{q'}(\T^d))$.
In particular, if we restrict on the time interval $[0,\tau]$, $\rho$ is still a local solution up to time $\tau$ and we can assume w.l.o.g. that $\rho\in C_w([0,\tau];L^{q'}(\T^d))$.

Since $\rho\in C_w([0,\tau];L^{q'}(\T^d))$, there exists $n\in \N$, depending on $\tau$, such that $\norm{\rho}_{L^{\infty}([0,\tau];L^{q'}(\T^d))}\le n$.

Finally, observe that $\bar\rho \not\equiv  0$ implies there exists $\varphi\in C^{\infty}(\T^d)$ such that $\int \varphi \bar \rho\neq 0$,
hence by density of $\mathcal{G}$ in $L^q$ we can find $\psi\in \mathcal{G}$ and $k\in \N$ such that $\left|\int\psi\bar\rho\right| \ge1/k.$

\medskip

\textit{\underline{Claim 2}. For every $\tau, k,n,\psi$, the set $\E_{\tau,k,n,\psi}$ is closed. } \label{claim2-thm-a}

\smallskip

Let $\{u_j\}_{j\in \N}\subseteq \E_{\tau,k,n,\psi}$ be such that $u_j\to u \in L^1_tL^q_x$. We want to prove $u\in\E_{\tau,k,n,\psi}$.\\
For any $j\in\N$, let $\bar\rho_j\in L^{q'}(\T^d)$ and $\rho_j\in C_w([0,\tau];L^{q'}(\T^d))$ solving
\begin{equation*}
    \left\{\begin{aligned}
        &\partial_t\rho_j + \div(\rho_ju_j)=0, \quad \text{on }\, \T^d\times [0,\tau]\\
        &\rho_j|_{t=0}=\bar\rho_j
    \end{aligned}\right.
\end{equation*}
and satisfying properties
(\textit{i}) and (\textit{ii}) as in \eqref{eq:def_iniseme_esistenza}.
By Lemma \ref{lemma:equicont_solutions}, for all $j\in \N$ and for every $\tau\ge t\ge s\ge 0$ it holds
    \begin{equation*}
        d_n(\rho_j(t),\rho_j(s))\le n \int_s^t\norm{u_j(z,\cdot)}_{L^q}\,dz.
    \end{equation*}
    In particular, since $u_j\to u$ in $L^1_tL^q_x$, we get that $\{t\mapsto \rho_j(t)\}_{j\in\N}$ is an equicontinuous family of maps with values in $\bar B_n^{q'}$, thus by Ascoli-Arzelà theorem it is precompact in $C_w([0,\tau];{\bar B_n^{q'}})$.
We conclude there exists subsequences $\{\bar \rho_{j_m}\}_{m\in\N}$, $\{\rho_{j_m}\}_{m\in\N}$ such that
\begin{equation*}
\begin{aligned}
    &\rho_{j_m}\to \rho \,\,\text{ in } \, C_w([0,\tau];{\bar B_n^{q'}})\\
    &\bar\rho_{j_m} \to \bar\rho \,\,\text{ in }\, L^{q'}(\T^d).
\end{aligned}
\end{equation*}
In particular,
\begin{equation*}
   \rho|_{t=0}=\bar\rho.
\end{equation*}
Clearly $\bar\rho$ still satisfies (\textit{i}). Finally the strong convergence $u_j \to u$ in $L^1_t L^q_x$ together with the weak* convergence $\rho_{j_m} \to \rho$ implies that 
$\bar \rho$ satisfies (\textit{ii}) as well.

\medskip

\textit{\underline{Claim 3}. For any $\tau, k,n,\psi$, the set $\E_{\tau,k,n,\psi}$ has empty interior.}

\smallskip

Observe that the complement of $\E_{\tau,k,n,\psi}$ is
\begin{equation*}
\begin{aligned}
    \E_{\tau,k,n,\psi}^c:= \big\{u\in L^1_tL^q_x\,:\, &\forall \bar \rho \in L^{q'}\,\text{ s.t. } 
      \,\, \left|\int_{\T^d}\psi(a)\bar\rho(a)\,da\right|\ge \frac{1}{k},\\
    &\text{ if } \rho\in C_{w}([0,\tau];L^{q'}(\T^d)) \,\,\text{ is a local solution to \eqref{eq:PDE} }\\
    &\text{ up to time $\tau$ with initial datum $\bar\rho$,} \\
    & \text{ then} \sup_{t\in [0,\tau]}\norm{\rho(t)}_{L^{q'}}>n\big\} .
    \end{aligned}
\end{equation*}
To prove the claim it is enough to show that $\E_{\tau,k,n,\psi}^c$ is dense in $L^1_tL^q_x$.
Therefore, by density of smooth vector fields in $L^1_tL^q_x$, it is enough to show that for  every vector field $u\in C^{\infty}([0,T]\times\T^d)$ and for every $\e>0$ there exists $w\in \E_{\tau,k,n,\psi}^c$ such that $\norm{u-w}_{L^{1}_tL^q_x}\le \e$.
But here it is enough to apply Proposition \ref{prop:main_proposition_existence} to $u$ with parameters $\e>0$, $\tau$, $M=n$, $\delta=1/k$ and $\varphi=\psi$ to construct such $w$ and conclude the proof.
\end{proof}

\subsection{Proof of the main Proposition \ref{prop:main_proposition_existence}}\label{sec:main_prop_existence}
The proof will proceed as follows.
\begin{enumerate}
    \item We will start by constructing in Lemma \ref{l:auton_bb} the smooth building block of the whole proof. 
    \item In Lemma \ref{l:bb_auton_zero} we will produce, by means of an appropriate rescaling, a smooth and arbitrarily small vector field having the property \eqref{eq:norm_inflation}.
    \item Finally, we will combine Lemma \ref{l:bb_auton_zero} together with Lemma \ref{lemma:composition_lemma} to prove Proposition \ref{prop:main_proposition_existence}.
\end{enumerate}

In the following, for a given smooth vector field $v$ we indicate the Jacobian of its flow map by 
\begin{equation*}
    J^v(t,x):=\mathrm{det }\nabla X^v(t,x).
\end{equation*}
We adopt also the notation $$Q_L(0)=\{x\in\R^d\,:\,|x_i|< L/2 \text{ for } i =1,\dots, d\}$$
for the open cube of side $L$ centered in the origin, in analogy with the standard notation $B_L(0)$ for the open ball of radius $L$ centered in the origin.
Finally, for a Borel set $A$ (on $\T^d$ or $\R^d$) we denote by $|A|$ its Lebesgue measure (on $\T^d$ or $\R^d$).

\begin{lemma}\label{l:auton_bb}
    For every $\eta\in(0,1)$ there is an autonomous vector field $ v_\eta \in C^\infty(\T^d; \R^d)$ such that
    \begin{equation}\label{eq:norm-v-eta}
        \norm{v_{\eta}}_{L^{\infty}}\le 1
    \end{equation}
    and 
    \begin{equation}\label{eq:jacob_bb}
        \big| \{ x\in \T^d\,:\,J^{v_{\eta}}(1, x) \geq \eta \} \big| \leq \eta.
    \end{equation}
\end{lemma}

\begin{proof}
    Let $\eta \in (0,1)$ be  fixed. We construct a vector field $v\in C^{\infty}_c(\R^d;\R^d)$ compactly supported in $Q_1(0)$ and depending on $\eta$ and then we define $v_{\eta}$  in the statement as the periodic extension of $v$.\\
    Let $\delta > 0$ to be fixed later depending on $\eta$.
    Let $\psi\in C^{\infty}_c(\R^d;\R)$ be a smooth, non negative cutoff with the following properties:
    \begin{equation*}
        \psi(x)=\left\{\begin{aligned}
            &0, && x\in \overline B_{\delta}(0),\\
            &1, && x\in Q_{1-2\delta}(0)\setminus \overline  B_{2\delta}(0),\\
            &0, &&\R^d\setminus Q_{1-\delta}(0).
        \end{aligned}\right.
    \end{equation*}
    If $\delta$ is small enough, $\overline{B}_{2\delta}(0) \subseteq Q_{1 - 2 \delta}$ and thus $\psi$ is well defined. 
    
    Define the smooth vector field $v\in C_c^{\infty}(\R^d)$ by
    \begin{equation*}
        v(x)=-\psi(x) \frac{\sqrt{d}}{2} \frac{x}{|x|}
    \end{equation*}
    and let $X:=X^v$ be its unique flow map and $J:=J^v$ the corresponding Jacobian.
    First of all we prove the following claim. 
    
    \medskip
    
    \textit{\underline{Claim}. $X(1,Q_{1-2\delta}(0))\subseteq \overline B_{2\delta}(0)$.}

    \smallskip

 \begin{figure}[t!]
    \centering
    \begin{subfigure}[b]{0.48\textwidth}
    \centering
    \includegraphics[width=1\linewidth]{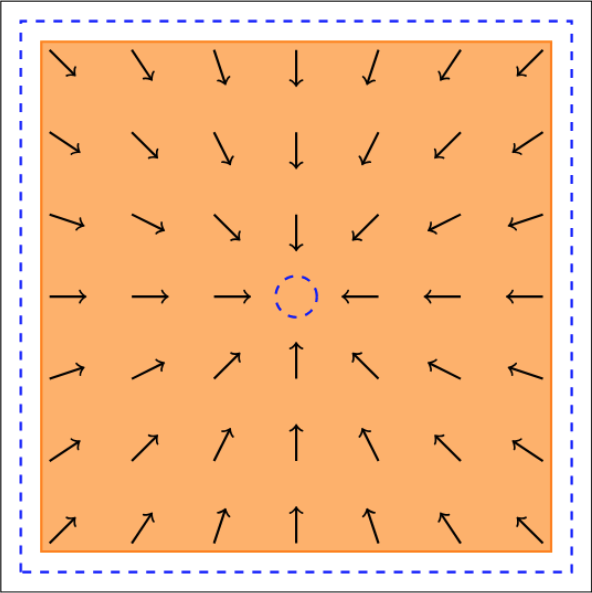}
    \caption{\label{fig:image1} $t=0$}
    \end{subfigure}
\quad
    \begin{subfigure}[b]{0.48\textwidth}
    \centering
    \includegraphics[width=1\linewidth]{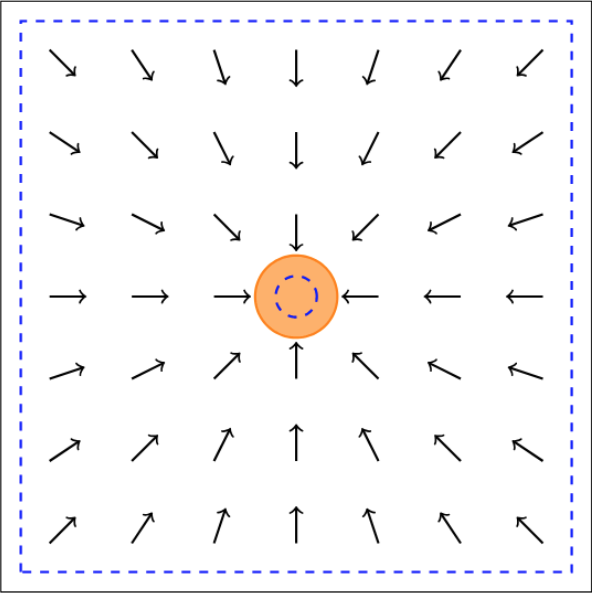}
    \caption{\label{fig:image2} $t=1$}
    \end{subfigure}
    \caption{A qualitative sketch of the stream lines of the autonomous vector field $v_{\eta}$ of Lemma \ref{l:auton_bb} restricted to $Q_1(0)$. 
    The region in between the two blue dashed curves is the support of $v_{\eta}$: the inner dashed curve is $\partial B_{\delta}(0)$, while the outer one is $\partial Q_{1-\delta}(0)$. 
    The orange square in figure \eqref{fig:image1} represents $Q_{1-2\delta}(0)$ at time $t=0$. 
    In figure \eqref{fig:image2} it is represented the outcome of the evolution at time $t=1$: the mass initially spread in $Q_{1-2\delta}(0)$ is now concentrated into $\bar B_{2\delta}(0)$.}\label{fig1}
\end{figure} 


    Observe that for every $x \in Q_1(0)$ the map $t\mapsto |X(t,x)|$ is non increasing. 
    In fact: $v(0) = 0$, therefore $X(t,0) \equiv 0$ for all $t\ge 0$; on the other hand, if $x\neq 0$ then $t\mapsto |X(t,x)|$ is smooth, $|X(t,x)|>0$ for every $t\ge 0$ and we can explicitly compute
     \begin{equation}\label{eq:traj_decrease_bb_auton}
    \begin{aligned}
         \partial_t |X(t,x)| &= - \frac{\sqrt{d}}{2} \frac{X(t,x)}{|X(t,x)|}\cdot \frac{X(t,x)}{|X(t,x)|} \psi(X(t,x))\\
         &=- \frac{\sqrt{d}}{2} \psi(X(t,x))\le0.
    \end{aligned}
    \end{equation}
    
    Next, notice that for every $x\in Q_{1-2\delta}(0)\setminus \overline  B_{2\delta}(0)$ it holds 
    \begin{equation}\label{eq:traj_main_bb_auton}
        X(t,x)=\frac{x}{|x|}\left (|x|-  \frac{\sqrt{d}}{2} t \right), \quad \forall t \in \left[0,\frac{2}{\sqrt{d}} \left(|x|-2\delta\right) \right].
    \end{equation}
    Indeed, for $x \in Q_{1-2\delta}(0)\setminus \overline  B_{2\delta}(0)$, $u(x) = - x/|x|$ and thus \eqref{eq:traj_main_bb_auton} follows by explicit computation. 
    
    Using now the fact that $t \mapsto |X(t,x)|$ is non-increasing (see \eqref{eq:traj_decrease_bb_auton}) we deduce that 
    \begin{equation}
        \label{eq:bb_ball}
        X(t,\overline B_{2\delta}(0))\subseteq \overline B_{2\delta}(0), \quad \forall t\ge 0.
    \end{equation}
    Similarly, using in addition \eqref{eq:traj_main_bb_auton} (at time $t =\frac{2}{\sqrt{d}} (|x| - 2\delta)$), we deduce that 
    \begin{equation}
        \label{eq:bb_cornice}
        \forall x\in Q_{1-2\delta}(0)\setminus \overline  B_{2\delta}(0), \quad  X(t,x) \in \overline{B}_{2\delta} \quad \forall t\ge \frac{2}{\sqrt{d}} \left( |x|-2\delta \right).
    \end{equation}
    Combining \eqref{eq:bb_ball} and \eqref{eq:bb_cornice}, we get 
    \begin{equation*}
        X(1,Q_{1-2\delta}(0))= X(1,\overline B_{2\delta}(0))\,\cup \,X(1,Q_{1-2\delta}(0)\setminus \overline  B_{2\delta}(0))\subseteq \overline{B}_{2\delta}(0)
    \end{equation*}
    which concludes the proof of the Claim.

    \medskip

    Set now $$A_{\eta}:=\{x\in Q_1(0)\,:\, J(1,x)\ge \eta\}.$$
    To estimate its Lebesgue measure, observe that  
    \begin{equation*}
    A_{\eta}\subseteq (Q_1(0)\setminus Q_{1-2\delta}(0))\cup (A_{\eta}\cap Q_{1-2\delta}(0))
    \end{equation*}
    First of all notice that 
    \begin{equation*}
        |Q_1(0)\setminus Q_{1-2\delta}(0)|\le C_d\, \delta,
    \end{equation*}
    where $C_\delta$ is a dimensional constant. 
    Let us now estimate the measure of the set 
    $$A_{\eta}\cap Q_{1-2\delta}(0)=\{x \in Q_{1-2\delta}(0)\,:\, J(1,x)\ge \eta \}.$$
    From a direct calculation we get
        \begin{equation*}
         \int_{Q_{1-2\delta}(0)}J(1,x)\, dx= \int_{X(1,Q_{1-2\delta}(0))}1\, dx' \le |B_{2\delta}(0)| \le C_d\delta^{d},
     \end{equation*}
    where in the first equality we  used the change of variable $x' = X(t,x)$, in the first inequality we used the Claim, and $C_d$ denotes a dimensional constant. Next we use Chebyscev inequality to compute 
   \begin{equation*}
       |A_{\eta}\cap Q_{1-2\delta}(0)|=|\{x \in Q_{1-2\delta}(0)\,:\, J(1,x)\ge \eta \}| \le C_d \frac{\delta^d}{\eta}.
   \end{equation*}
    Overall we have 
    \begin{equation*}
        |A_{\eta}|\le |Q_1(0)\setminus Q_{1-2\delta}(0)| +|A_{\eta}\cap Q_{1-2\delta}(0)| \le C_d\left(\delta + \frac{\delta^d}{\eta}\right)\le \eta
    \end{equation*} 
    once $\delta>0$ is chosen sufficiently small (depending on $\eta$).
\end{proof}


As already mentioned before, we are now going to appropriately rescale the vector field $v_{\eta}$. In order to keep track on the set where the Jacobian of the rescaled vector field is large (as we did for $v_{\eta}$ in \eqref{eq:jacob_bb}), we will use the following standard identities. 

\begin{lemma}\label{l:rescaling-oscillation}
    Let $v\in C^{\infty}([0,\infty)\times\T^d;\R^d)$ and let $\lambda\in \N_{>0}$ and $\tau>0$. Define $v_{\tau,\lambda}(x):=\frac{1}{\tau\lambda}v\left(\frac{t}{\tau},\lambda x\right)$. 
    Then 
    \begin{equation}
    \label{eq:rescaled_flux}
    \begin{aligned}
        X_{{\tau,\lambda}}(t,x) &=\frac{1}{\lambda}X\left(\frac{t}{\tau},\lambda x\right), &
        J_{{\tau,\lambda}}(t,x) & = J\left(\frac{t}{\tau},\lambda x\right),
    \end{aligned}
    \end{equation}
where 
\begin{equation*}
    \begin{aligned}
        X &:= X^v, & J & := J^v, &
        X_{\tau, \lambda} & := X^{v_{\tau, \lambda}}, & J_{\tau, \lambda} & := J^{v_{\tau, \lambda}},
    \end{aligned}
\end{equation*}
denotes the flow map and the Jacobian determinant of $v$ and $v_{\tau, \lambda}$ respectively.
\end{lemma}

We are now ready to prove the following lemma, which is essentially the statement of Proposition \ref{prop:main_proposition_existence} with $u \equiv 0$. 

\begin{lemma}\label{l:bb_auton_zero}
    For every $\e>0$, $\delta>0$, $\tau>0$, $M>1$ and $\varphi\in C^\infty(\T^d)$, there exists an autonomous vector field $v=v_{\e, \delta, \tau,M,\varphi}\in C^{\infty}(\T^d;\R^d)$ such that
    \begin{equation}
    \label{eq:norm-v}
        \norm{v}_{L^{\infty}}\le \e
    \end{equation}
    and having the following property. For every $\bar\rho \in L^{q'}(\T^d)$ such that
    $$\left| \int_{\T^d}\varphi(x)\bar\rho(x)\,dx \right|\ge \delta,$$
    the unique solution $\rho\in C_w([0,+\infty);L^{q'}(\T^d))$ to the Cauchy problem
    \begin{equation*}
       \left\{\begin{aligned}
            &\partial_t\rho + \div (\rho v)=0,\\
            &\rho_{t=0}=\bar\rho
       \end{aligned}\right.
    \end{equation*}
    satisfies
    \begin{equation*}
         \sup_{s\in[0,\tau]}\norm{\rho(s)}_{L^{q'}}> M.
     \end{equation*}
\end{lemma}

\begin{proof}
    
    Let $\e>0$, $\delta>0$, $\tau>0$, $M>1$ and $\varphi\in C^\infty(\T^d)$ be fixed. Let $\eta>0$ and $\lambda \in \N$ be two parameters which will be fixed during the proof, depending on $\e, \delta, \tau, M$, and $\varphi$.   Let $v_{\eta}\in C^{\infty}(\T^d;\R^d)$ be the vector field given by Lemma \ref{l:auton_bb}. Set
    \begin{equation*}
        v_{\eta,\tau,\lambda}(x):=\frac{1}{\tau \lambda}v_{\eta}(\lambda x).
    \end{equation*}

\begin{figure}[t]
    \label{fig2}
    \centering
    \begin{subfigure}[b]{0.48\textwidth}
    \centering
    \includegraphics[width=1\linewidth]{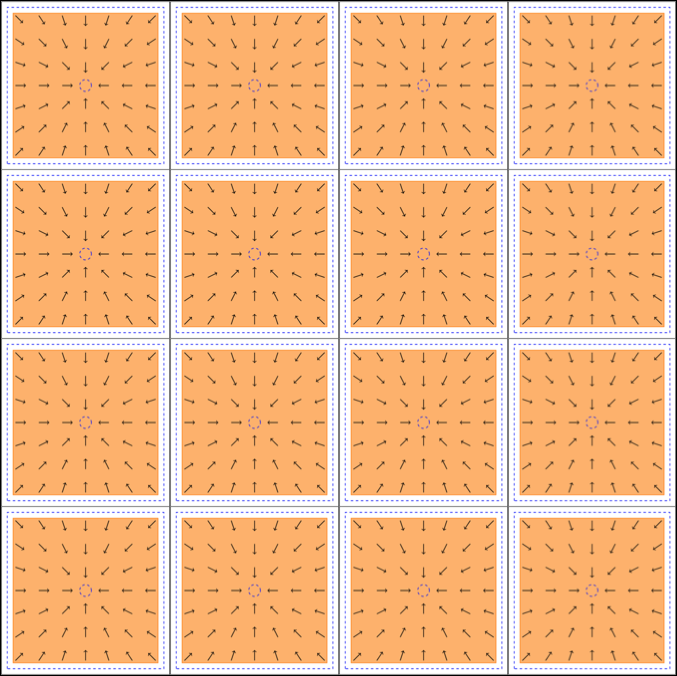}
    \caption{\label{fig:image3} $t=0$}
    \end{subfigure}
\quad
    \begin{subfigure}[b]{0.48\textwidth}
    \centering
    \includegraphics[width=1\linewidth]{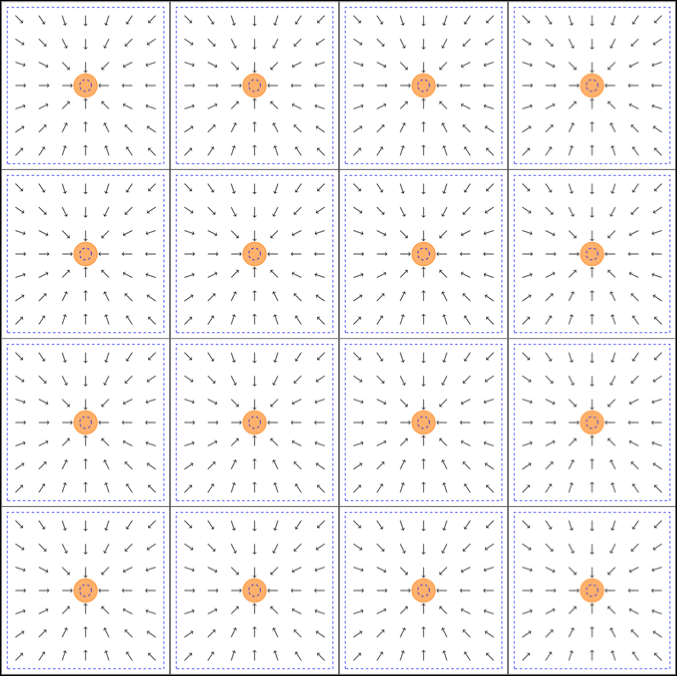}
    \caption{\label{fig:image4} $t=\tau$}
    \end{subfigure}
    \caption{The vector field $v$ in Lemma \ref{l:bb_auton_zero} is a rescaled version of $v_{\eta}$ from Lemma \ref{l:auton_bb} that replicates the same behavior at different space-time scales.}
\end{figure} 

    \noindent We denote by
    \begin{equation*}
        \begin{aligned}
            X_\eta & := X^{v_{\eta}}, &
            \Phi_\eta & := \Phi^{v_{\eta}}, &
            J_\eta & := J^{v_{\eta}}, \\
            X_{\eta, \tau, \lambda} & := X^{v_{\eta, \tau, \lambda}}, &
            \Phi_{\eta, \tau, \lambda} & := \Phi^{v_{\eta, \tau, \lambda}}, &
            J_{\eta, \tau, \lambda} & := J^{v_{\eta, \tau, \lambda}}, &
        \end{aligned}
    \end{equation*}
    the flow map, the inverse flow map and the Jacobian determinant associated to $v_\eta$ and $v_{\eta, \tau, \lambda}$ respectively.

    \medskip
   
    \textit{Step 1 (choice of $\lambda$ and proof of \eqref{eq:norm-v}).} Notice that, by property \eqref{eq:norm-v-eta} in Lemma \ref{l:auton_bb},
    \begin{equation}
    \label{eq:est_auton_infty}
        \norm{v_{\eta,\tau,\lambda}}_{L^{\infty}}\le \frac{1}{\tau\lambda} < \e,
    \end{equation}
    once $\lambda$ is chosen sufficiently large. This fixes the choice of $\lambda$ (depending on $\tau, \e$) and guarantees   that $v_{\eta, \tau, \lambda}$ satisfies \eqref{eq:norm-v}.

 \medskip
    
    \textit{Step 2 (measure of $A_{\eta, \tau, \lambda})$}
     Set $A_{\eta, \tau,\lambda}  :=\left\{x\in\T^d\,:\,J_{{\eta, \tau,\lambda}}(\tau,x)\ge \eta \right\}.$

    We claim that  
    \begin{equation}
        \label{eq:measure-A}
        |A_{\eta, \tau, \lambda}| \leq \eta.
    \end{equation}

    \noindent We argue as follows: using \eqref{eq:rescaled_flux}, we get
    \begin{equation*}
    \begin{aligned}
        A_{\eta, \tau,\lambda} 
        & \ = \left\{ x \in \T^d \, : \, J_\eta(1, \lambda x ) \geq \eta \right\}.
    \end{aligned}
    \end{equation*}
    Observe now that for any measurable $f : \T^d \to \R$, 
    \begin{equation*}
        \left| \left\{ x \in \T^d \ : \ f(\lambda x) \geq \eta \right\} \right| = \left| \left\{ x \in \T^d \ : \ f(x) \geq \eta \right\} \right|. 
    \end{equation*}
    Therefore
    \begin{equation}
    \label{eq:jacob_rescaled_bb}
        \begin{aligned}
            |A_{\eta, \tau, \lambda}| 
            & = \left| \left\{ x \in \T^d \ : \ J_\eta(1, \lambda x) \geq \eta \right\} \right| 
             = \left| \left\{ x \in \T^d \ : \ J_\eta(1, x) \geq \eta \right\} \right| \leq \eta,
        \end{aligned}
    \end{equation}
    by  \eqref{eq:jacob_bb}. 
    
    \medskip

    \textit{Step 3 (norm inflation)}.     Let  $\bar \rho \in L^{q'}(\T^d)$ be such that 
    \begin{equation*}
        \left|\int_{\T^d}\bar\rho (x)\varphi(x)\,dx\right|\ge \delta.
    \end{equation*}
    Let $\rho\in C_{w}([0,+ \infty);L^{q'}(\T^d))$ be the unique solution to 
    \begin{equation*}
        \left\{\begin{aligned}
            \partial_t \rho + \div (\rho \,v_{\eta,\tau,\lambda})=0,\\
            \rho|_{t=0}=\bar\rho.
        \end{aligned}
        \right.
    \end{equation*}
    We claim that, for a suitable choice of $\eta>0$,  
    \begin{equation}
        \label{eq:norm-inflation}
        \sup_{s \in [0, \tau]} \|\rho(s)\|_{L^{q'}} > M.   
    \end{equation}
    
    We argue as follows.     W.l.o.g. we can assume  $\norm{\bar\rho}_{L^{q'}_x} \leq M$, otherwise \eqref{eq:norm-inflation} is trivially satisfied. We thus have
    \begin{equation}
    \label{eq:jacob_rescaled_main_cpt}
        \begin{aligned}
            \delta &\le \left|\int_{\T^d}\bar\rho (x)\varphi(x)\,dx\right|\le \left|\int_{A_{\eta,\tau,\lambda}}\bar\rho(x)\varphi(x)\, dx\right| + \left|\int_{A_{\eta,\tau,\lambda}^c}\bar\rho(x)\varphi(x)\, dx\right| \\
            & \text{(H\"older inequality and change of variable $x=\Phi_{\eta, \tau,\lambda}(\tau,x')$)} \\
            &\le \| \bar \rho\|_{L^{q'}(\T^d)} \norm{\varphi}_{L^{q}(A_{\eta,\tau,\lambda)}} + \int_{X_{\eta, \tau,\lambda}(\tau,A_{\eta,\tau,\lambda}^c)}|\rho(\tau,x')\varphi(\Phi_{{\eta, \tau,\lambda}}(\tau,x'))|\,dx' \\
            & \leq M \norm{\varphi}_{L^{q}(A_{\eta,\tau,\lambda)}} 
            + \norm{\rho(\tau)}_{L^{q'}(\T^d)} 
            \| \varphi (\Phi_{\eta, \tau, \lambda}(\tau, \cdot))\|_{L^q(X_{\eta, \tau,\lambda}(\tau,A_{\eta,\tau,\lambda}^c))}. 
        \end{aligned}
    \end{equation}
    Observe now that 
    \begin{equation*}
        \begin{aligned}
            \| \varphi (\Phi_{\eta, \tau, \lambda}(\tau, \cdot))\|_{L^q(X_{\eta, \tau,\lambda}(\tau,A_{\eta,\tau,\lambda}^c))}
            & = \left(\int_{X_{\eta, \tau,\lambda}(\tau,A_{\eta,\tau,\lambda}^c)} \left| \varphi(\Phi_{\eta, \tau, \lambda}(\tau, x')) \right|^q dx' \right)^{\frac{1}{q}}
            \\
            \text{(change of variable $x=\Phi_{\eta, \tau,\lambda}(\tau,x')$)}
            & =\left( \int_{A_{\eta,\tau,\lambda}^c}|\varphi(x)|^qJ_{{\eta,\tau,\lambda}}(\tau,x)\,dx\right)^{\frac{1}{q}}\\
            \text{(since $J_{\eta, \tau, \lambda}(\tau, \cdot) \leq \eta$ on $A_{\eta, \tau, \lambda}^c$)}
            &\le \eta^{\frac{1}{q}} \norm{\varphi}_{L^{q}(\T^d)}.
        \end{aligned}
    \end{equation*}
    Therefore, continuing the estimate in \eqref{eq:jacob_rescaled_main_cpt}, we get
    \begin{equation*}
        \delta \leq M \|\varphi\|_{L^q(A_{\eta, \tau, \lambda})} + \eta^{\frac{1}{q}} \| \varphi\|_{L^q(\T^d)} \| \rho(\tau)\|_{L^{q'}(\T^d)},
    \end{equation*}
    from which we get 
    \begin{equation}
    \label{eq:bdd-below}
        \| \rho(\tau)\|_{L^{q'}} \geq \frac{\delta - M \| \varphi\|_{L^q(A_{\eta, \tau, \lambda})}}{\eta^{\frac{1}{q}} \| \varphi\|_{L^q(\T^d)}}. 
    \end{equation}
    We observe now that $\| \varphi\|_{L^q(A_{\eta, \tau, \lambda})} \to 0$ as $\eta \to 0$ by the absolute continuity of the integral, and the fact that $|A_{\eta, \tau, \lambda}|  \to 0$ as $\eta \to 0$, by \eqref{eq:jacob_rescaled_bb}. Therefore the r.h.s. in \eqref{eq:bdd-below} converges to $+\infty$ as $\eta \to 0$. We can thus choose $\eta$ small enough such that 
    \begin{equation*}
        \| \rho(\tau)\|_{L^{q'}} > M. 
    \end{equation*}
    This fixes the choice of $\eta$, depending on $\delta, \tau, M$ and $\varphi$ and concludes the proof of \eqref{eq:norm-inflation}.

    With such choices of $\eta$ and $\lambda$, the vector field $v:=v_{\eta,\tau,\lambda}$ has all the desired properties.
\end{proof}


We are now ready to prove Proposition \ref{prop:main_proposition_existence}.

\begin{proof}[Proof of Proposition \ref{prop:main_proposition_existence}]
    Let $u$, $\e, \delta, \tau, M$ and $\varphi$ be as in the statement. Let $\tilde \e>0$ and $\tilde M>0$ be two constants to be fixed later. Let  $v=v_{\tilde \e, \delta, \tau,\tilde M,\varphi} \in C^{\infty}(\T^d;\R^d)$ be the autonomous vector field given by Lemma \ref{l:bb_auton_zero} (notice that we use $\tilde \e$ and $\tilde M$ instead of $\e$ and $M$).
    Define the smooth vector field
    \begin{equation*}
        w(t,x) := u(t,x) + (\nabla \Phi^u(t,x))^{-1}v(\Phi^u(t,x))
    \end{equation*}
    Fix $\tilde \e>0$ such that
    \begin{equation*}
        \begin{aligned}
        \norm{(\nabla \Phi^u)^{-1}v(\Phi^u)}_{L^\infty(0,T;L^q(\T^d))}
        & \le \| (\nabla \Phi^u)^{-1}\|_{L^\infty([0, T] \times \T^d)} \| v\|_{L^\infty(\T^d)} 
         \le C_u \tilde \e  \leq \e 
        \end{aligned}
    \end{equation*}
    where, in the second inequality, we used \eqref{eq:norm-v}.     Let now $\bar\rho\in L^{q'}$ be such that 
    $$\left| \int_{\T^d}\varphi(x)\bar\rho(x)\,dx \right|\ge \delta.$$
    Since $w$ is smooth, there exists a unique $\rho \in C_{w}([0,+\infty);L^{q'}(\T^d))$ solution to
     \begin{equation*}
       \left\{\begin{aligned}
            &\partial_t\rho + \div (\rho w)=0,\\
            &\rho_{t=0}=\bar\rho.
       \end{aligned}\right.
    \end{equation*}
    Therefore, by Lemma \ref{lemma:composition_lemma} and the fact that $\rho$ is the \emph{unique} solution, the representation formula  
    \begin{equation}\label{eq:expl_form_dens_ex}
        \rho(t)=X^u(t)_\sharp\tilde\rho(t) \text{ for all } t \in [0,T]
    \end{equation}
    must hold, where $\tilde\rho \in C_{w}([0,+\infty);L^{q'}(\T^d))$ is the unique solution to 
    \begin{equation*}
       \left\{\begin{aligned}
            &\partial_t\tilde\rho + \div (\tilde\rho v)=0,\\
            &\tilde\rho_{t=0}=\bar\rho.
       \end{aligned}\right.
    \end{equation*}
    From \eqref{eq:expl_form_dens_ex}, applying $\Phi^u(t)_\sharp$ on both sides, we get also
    \begin{equation}
        \label{eq:expl_form_dens_ex_inv}
        \Phi^u(t)_\sharp \rho(t) = \tilde \rho(t) \text{ for all } t \in [0,T].
    \end{equation}
    By using the properties of $v$ coming from Lemma \ref{l:bb_auton_zero}
    and the explicit expression \eqref{eq:expl_form_dens_ex_inv} we can estimate
    \begin{equation*}
        \begin{aligned}
            \tilde M < \sup_{s \in [0,\tau]} \| \tilde \rho(s)\|_{L^{q'}} \leq C_u \sup_{s \in [0,\tau]|} \| \rho(s)\|_{L^{q'}}
        \end{aligned}
    \end{equation*}

  \noindent Therefore we can choose $\tilde M>0$ large enough in such a way that
    \begin{equation*}
        \sup_{s\in[0,\tau]}\norm{\rho(s)}_{L^{q'}}\ge \frac{\tilde M}{C_u} > M,
    \end{equation*}
    and this concludes the proof of Proposition \ref{prop:main_proposition_existence} and thus also of Theorem \ref{thm:existence-lq}. 
\end{proof}

\section{Residuality of uniqueness}\label{s:residuality_uniqueness}
 To prove Theorem \ref{thm:uniq-resid} we will show that $\X\cap \U_q$ is contained into a countable union of closed sets with empty interior (w.r.t. the topology of $(\X, d_\X)$), as we did in the proof of Theorem \ref{thm:existence-lq}.
The argument will nevertheless be ``more direct'', and thus easier. The proof will be divided into three steps, as in Theorem \ref{thm:existence-lq}: both the first about showing the inclusion into the countable union of sets and the second about their closure will look indeed very similar to Claims 1 and 2 of Theorem \ref{thm:existence-lq}. On the other hand, the proof that these sets have empty interior will follow directly from the assumption that $\U_q\cap \X$ is dense, and no further construction will be needed.

\begin{proof}[Proof of Theorem \ref{thm:uniq-resid}]
    First of all notice that, by linearity and the fact that $\rho \equiv 0$ is always a solution with zero initial datum, we can rewrite $ \X\cap\U_q$ as follows:
    
    \begin{equation}
    \begin{aligned}
        \label{eq:charact-uq}
        \X\cap\U_q=\big\{u\in \X\,:\, \forall\tau\in(0,T],\,\,\exists!\rho \in L^{\infty}_{loc}([0,\tau);L^{q'}(\T^d))\,\,\\\ \text{local solution to \eqref{eq:PDE} with } \rho|_{t=0}=0 \big\}.
    \end{aligned}
    \end{equation}

 We prove that the complement is a first category set in $\X$. 
 
    \medskip
    
    \underline{\textit{Claim 1.}} \textit{
    The following inclusion holds. 
    \begin{equation*}
        \X\cap\mathcal{U}_q^c\subseteq \bigcup_{\tau \in \Q\cap (0,T]}\bigcup_{n\in\N} \bigcup_{k\in\N} F_{\tau,n,k}
    \end{equation*}
    where
    \begin{equation*}
        \begin{aligned}
              F_{\tau,n,k}:=\left\{u\in \mathcal{X}\,:\, \exists \rho\in C_w([0,\tau];L^{q'}(\T^d))\,:\,
	       \text{ local solution to \eqref{eq:PDE} with $\rho|_{t=0}=0$} \right.\\ 
           \text{s.t. }    \left.\norm{\rho}_{L^{\infty}([0,\tau];L^{q'}(\T^d))}\le n
           \,\text{ and }\,d_n(\rho(\tau),0)\ge\frac{1}{k} \right\}.
        \end{aligned}
    \end{equation*}
    and $d_n$ is the distance inducing the weak* topology on $\overline B_n^{q'}$ given by Lemma \ref{lemma:equicont_solutions}.
}

\smallskip
Let $u\in \X \cap \mathcal{U}_q^c$. By \eqref{eq:charact-uq}, there exists a time $\tau'\in(0,T]$ and at least one solution $\rho\in L^{\infty}([0,\tau'];L^{q'}(\T^d))$ that is not identically equal to zero. Clearly there is $n\in \N$ such that $\norm{\rho}_{L^{\infty}_tL^{q'}_x}\le n$. As we observed at the beginning of Section \ref{ss:ex:prelim}, we can assume w.l.o.g. that $\rho\in C_w([0,\tau'];L^{q'}(\T^d))$. Now, since $\rho \not \equiv 0$, there must exist a time $\tau \in (0,\tau']$ such that $\rho(\tau)$ is nonzero.
Therefore, since $\rho(\tau)\in\overline B_n^{q'}$, we can quantify the fact that $\rho(\tau)$ is nonzero by means of the distance $d_n$: more precisely, there must be $k\in\N$ such that $d_n(\rho(\tau),0)\ge 1/k$. Overall we showed that $u \in F_{\tau, n, k}$. 

\medskip

 \underline{\textit{Claim 2.}} \textit{For every $\tau,n,k$ the set
	$F_{\tau,n,k}$ is closed in $\mathcal{X}$.}

\smallskip

    The proof is very similar to that of Claim 2 in Theorem \ref{thm:existence-lq} on page \pageref{claim2-thm-a}. 
	Take $\{u_j\}_{j\in\N}\subseteq F_{\tau,n,k}$ such that $d_\X(u_j, u) \to 0$. Observe that, in particular, by Point \ref{it:distances} in the statement of Theorem \ref{thm:uniq-resid}, $\|u_j - u\|_{L^1_tL^q_x} \to 0$. We want to prove $u\in F_{\tau,n,k}.$
    
   By definition of $F_{\tau, n,k}$, for any $j\in\N$, there are $\bar\rho_j\in L^{q'}(\T^d)$ and $\rho_j\in C_w([0,\tau];L^{q'}(\T^d))$ solving
\begin{equation*}
    \left\{\begin{aligned}
        &\partial_t\rho_j + \div(\rho_ju_j)=0, \quad \text{on }\, \T^d\times [0,\tau]\\
        &\rho_j|_{t=0}=0
    \end{aligned}\right.
\end{equation*}
and satisfying 
$$\norm{\rho_j}_{L^{\infty}([0,\tau];L^{q'}(\T^d))}\le n
           \,\text{ and }\,d_n(\rho_j(\tau),0)\ge\frac{1}{k}.$$
	Arguing as in the proof of Claim 2 of Theorem \ref{thm:existence-lq}, we get a 
     a subsequence $\{\rho_{j_m}\}_{m\in \N}$ such that
    \begin{equation*}
        \rho_{j_m}  \to \rho\,\, \text{ in } C_w([0,\tau]; L^{q'}(\T^d)).
    \end{equation*}
    Both bounds $\norm{\rho}_{L^{\infty}([0,\tau]:L^{q'}(\T^d))}\le n$ and $d_n(\rho(\tau), 0)\ge 1/k$ are preserved by $C_wL^{q'}_x$ convergence.
    The proof that $\rho$ is a local weak solution with initial datum $\rho|_{t=0}=0$ follows immediately from the strong  convergence $u_j \to u$ in $L^1_t L^q_x$ and the weak* convergence $\rho_{j_m} \overset{*}{\rightharpoonup} \rho_j$. 
    
\medskip

\underline{\textit{Claim 3.}} \textit{For every $\tau,n,k$ the set
	$F_{\tau,n,k}$ has empty interior in $\mathcal{X}$.}

\smallskip

This is equivalent to prove that $\X\setminus F_{\tau,n,k}$ is dense in $\X$.
By hypothesis $\X\cap\mathcal{U}_q$ is dense in $\X$. 
Hence, density of $\X\setminus F_{\tau,n,k}$ comes from the following inclusion:
\begin{equation*}
\begin{aligned}
        \X\cap\mathcal{U}_q \subseteq  \X\setminus F_{\tau,n,k}
        =   \bigg\{ u\in\X\,:\,\text{if }\,\rho\in C_w([0,\tau];L^{q'}(\T^d))\,\text{ is a local solution} \\
         \text{to \eqref{eq:PDE} with $\rho|_{t=0} = 0$ and $\norm{\rho}_{L^{\infty}([0,\tau];L^{q'}(\T^d))}\le n$,} 
        \text{ then } d_n(\rho(\tau),0)<\frac{1}{k}&\bigg\}.
\end{aligned}
\end{equation*}
In fact, if $u\in \mathcal{U}_q $, for all $\tau\in [0,T]$ the only local solution on $[0,\tau]$ is $\rho\equiv 0,$. This concludes the proof of the claim and also of  Theorem \ref{thm:uniq-resid}. 
\end{proof}

\section{Density for the continuity equation}
\label{s:density-pde}

In this section we prove the density statement for the PDE, i.e. point (i) of Theorem \ref{thm:density}. The idea of the proof is the following:
\begin{enumerate}
    \item \label{pt:literature} The building block for our construction is provided by the incompressible vector field constructed in \cite{BrueColomboKumar2024}, for which uniqueness of solutions to the associated continuity equation does not hold (in a suitable class of densities).
    \item \label{pt:rescaling} By a simple rescaling argument we show that vector field in Point \eqref{pt:literature} can be made arbitrarily small.
    \item Using Point \eqref{pt:rescaling}, and the construction in Section \ref{s:push-forward}, we deduce that, arbitrarily close  to any smooth vector field, there is a vector field for which uniqueness of solutions to the continuity equation does not hold.
    \item Finally, we use the density of smooth maps in $\Y$ and the previous point to conclude the proof.
\end{enumerate}

\bigskip

We start by recalling the following result from \cite{BrueColomboKumar2024}.
\begin{theorem}[Bru\`e, Colombo, Kumar]
\label{brue_colombo_kumar_2}
    Let Let $q,p$ be as in \eqref{eq:bck-range}.
    There exists an incompressible vector field $\mathrm{v}\in C([0,1],W^{1,p}\cap L^{q}(\R^d,\R^d))$ such that 
    \eqref{eq:PDE} admits two distinct (non-negative) solutions $\rho^1, \rho^2 \in C([0,1];L^{q'}(\R^d))$, with the same initial datum. Moreover
    \begin{equation*}
        {\rm supp}\, \mathrm{v}(t, \cdot), \,{\rm supp}, \rho^1 (t,\cdot), \,{\rm supp}\,  \rho^2(t,\cdot) \subseteq Q_1(0)
    \end{equation*}
    for all $t \in [0,1]$.
\end{theorem}

We want to properly rescale and periodize the vector field and the corresponding solutions to obtain a family of vector fields with the non-uniqueness property that are moreover arbitrarily small in $L^{\infty}_t(L^q\cap W^{1,p})_x$ (and thus also in $L^1_t(L^q\cap W^{1,p})_x$). 

\begin{proposition}[Rescaled version of Theorem \ref{brue_colombo_kumar_2}]
\label{prop:bb_nonuniqu_continuity_sobolev}
    Let $q,p$ be as in \eqref{eq:bck-range}. For any $\varepsilon>0$ there exists an incompressible vector field $v\in C([0,1];L^{q}\cap W^{1,p}(\T^d;\R^d))$ such that 
    \begin{equation*}
        \norm{v}_{L^{\infty}_t(L^q\cap W^{1,p})_x}<\varepsilon
    \end{equation*}
    and $v$ admits two distinct (non-negative) distributional solutions to \eqref{eq:PDE} in the class $C([0,1];L^{q'}(\T^d;\R))$.
\end{proposition}

\begin{proof}

Let $\mathrm{v} \in C([0,T],W^{1,p}\cap L^{q}(\R^d,\R^d))$ be given by Theorem \ref{brue_colombo_kumar_2}, with the two densities $\rho^1, \rho^2$ satisfying 
\begin{equation*}
\rho^1|_{t=0} = \rho^2|_{t=0}, \qquad \rho^1 \not \equiv \rho^2. 
\end{equation*}
Let $\mu>1$ be a  \emph{concentration} parameter to be fixed later, depending on $\e$. 
Set
\begin{equation*}
        \mathtt v_{\mu}(t,x):=\frac{1}{\mu}\mathtt v(t,\mu x),\qquad \rho^i_{\mu}(t,x):=\rho^i(t,\mu x), \quad i=1,2.
\end{equation*}
The scaling guarantees that $\mathtt v_{\mu}$ and $\rho^i_\mu$, $i=1,2$, still solve \eqref{eq:PDE} with the (same) rescaled initial datum $\rho^i_{0,\mu}(x)=\rho^i_0(\mu x)$, but $\rho^1_\mu \not \equiv \rho^2_\mu$.
A simple computation shows that
 \begin{equation}
 \label{eq:scaling_concentration} 
     \begin{aligned}
            \norm{\mathtt v_{\mu}}_{L^{\infty}_tL^q_x}=\frac{1}{\mu^{1+\frac{d}{q}}}\norm{\mathtt v}_{L^{\infty}_tL^q_x},\qquad
                 \norm{\nabla \mathtt v_{\mu}}_{L^{\infty}_tL^p_x}=\frac{1}{\mu^{\frac{d}{p}}}\norm{\nabla \mathtt v}_{L^{\infty}_tL^p_x}.
    \end{aligned}
 \end{equation}
Moreover the supports of $\mathtt v_\mu$, $\rho^1_\mu$, $\rho^2_\mu$ are contained in $Q_{\frac{1}{\mu}}(0)$. We can thus extend them by periodicity on $\T^d$. By picking $\mu$ large enough (depending on $\varepsilon$) we get the conclusion. 
\end{proof}

\subsection{Proof of Point (i) of Theorem \ref{thm:density}}
\label{ss:proof-i}
    We can assume w.l.o.g. that $T=1$. The general case follows from the case $T=1$ by a simple rescaling argument $t \mapsto t/T$, $u \mapsto u/T$.
 
    By density of smooth maps in $\Y$
    it is enough to show that for every $u\in C^{\infty}([0,1]\times \T^d;\R^d)$  with $\div u=0$ and for every $\e>0$, there exists another incompressible vector field $w\in C([0,1];L^q\cap W^{1,p}(\T^d;\R^d))$, with $q$ and $p$ as in \eqref{eq:bck-range}, such that 
    \begin{equation*}
        \norm{u-w}_{L^{\infty}_t(L^q\cap W^{1,p})_x} < \varepsilon
    \end{equation*}
    and $w$ admits two distinct solutions to \eqref{eq:PDE} in the class $L^{\infty}_tL^{q'}_x$ starting from the same initial datum.\\
     By Proposition \ref{prop:bb_nonuniqu_continuity_sobolev}, we can find an incompressible vector field $v\in C_t(L^q\cap W^{1,p})_x$ admitting two distinct solutions $\rho^1,\rho^2\in C_tL^{q'}_x$ to \eqref{eq:PDE} starting from the same initial datum $\rho_0\in L^{q'}$ and such that $\|v\|_{L^{\infty}_t(L^q\cap W^{1,p})_x}$ can be made arbitrarily small. In particular, we can choose $v$ such that
    \begin{equation}
        \label{eq:smallness-perturb}
        \left\| (\nabla\Phi^u(t,x))^{-1}v(t,\Phi^u(t,x)) \right\|_{L^{\infty}_t(L^q\cap W^{1,p})_x} < \varepsilon. 
    \end{equation}
    Set
    \begin{equation*}
        w(t,x):=u(t,x) + (\nabla\Phi^u(t,x))^{-1}v(t,\Phi^u(t,x))
    \end{equation*}
    Since $\div u = 0$, by Remark \ref{rmk:div-u-equal-div-w} the vector field $w$ is incompressible as well. Moreover, by \eqref{eq:smallness-perturb} it holds $\norm{u-w}_{L^{\infty}_t(L^q\cap W^{1,p})_x} < \varepsilon$.
    Notice now that, thanks to item (i) of Lemma \ref{lemma:composition_lemma} and the fact that $X^u(t, \cdot)$ is a diffeomorphism, the densities $$\tilde \rho^i(t,x):=X^u(t, \cdot)_\sharp \rho^i(t, \cdot), \quad i=1,2$$
    are two distinct distributional solutions to the continuity equation with vector field $w$ and with the same initial datum, so that $w$ has all the desired properties. \qed


\section{Density for the ODE}
\label{s:density-ode}

In this section we prove the density statement for the ODE, namely Point (ii) of Theorem \ref{thm:density}.
The strategy of the proof is similar to the one described at the beginning of Section \ref{s:density-pde} for the density statements for the PDE, with one substantial difference (Point \eqref{pt:rescaling} below):
\begin{enumerate}
    \item \label{pt:literature-ode} In \cite{kumar2023} the author constructs (for every $p<d$ and $\alpha < 1$) an incompressible vector field $v$ belonging to the space
    \begin{equation*}
        C_t W^{1,p}_x \cap C^\alpha_{t,x}
    \end{equation*}
    such that $A_v$ as defined in \eqref{eq:non-uniq-ode} has full measure.
    \item \label{pt:rescaling} As in Section \ref{s:density-pde}, we need to produce arbitrarily small 
    (in the desired topology, namely $C_t W^{1,p}_x \cap C_{t,x}$) 
    vector fields $u$ having the property that $A_{u}$ is a full measure set. 
    However, differently than the analysis in Section \ref{s:density-pde}, here rescaling arguments do not work. 
    Indeed, rescaling ``by oscillation'' (i.e. as in Lemma \ref{l:rescaling-oscillation}) would preserve the property of a.e. non uniqueness of integral curves but would not make the rescaled vector field small in $W^{1,p}$, whereas 
    rescaling ``by concentration'' (as we did for the PDE in Section \ref{s:density-pde}) would achieve the smallness in the desired topology but would not preserve the a.e. non uniqueness property.
    Nevertheless, we show how to easily adapt the construction in \cite{kumar2023} to obtain the desired smallness. 
    \item Finally, as in Section \ref{s:density-pde}, we combine Point \eqref{pt:rescaling},  the construction in Section \ref{s:push-forward} and the density of smooth maps in $\mathcal{Z}$ to conclude the proof of the density. 
\end{enumerate}

\bigskip

For the flow map associated to a (smooth) vector field $u \in C^\infty([0,T] \times \T^d)$, we used so far the notation $X^u(t,x)$ (see \eqref{eq:flow-map}). In what follows, it will be useful to adopt, when needed, also the following (standard) notation for flow maps, where the \emph{starting time} is explicitly written: we indicate by $X^u:[0,T]\times [0,T]\times \T^d\to \T^d$ the solution of the ODE
\begin{equation*}
\left\{
    \begin{aligned}
         \partial_t X^u(t,s,x)&=u(t,X^u(t,s,x)),\\
         X^u(s,s,x)&=x,
    \end{aligned}\right.
\end{equation*}
and we call it \emph{the flow with starting time} associated to $u$. This notation is useful to exploit the semigroup property: 
\begin{equation}
    \label{eq:semigroup}
    X^u(t_3,t_1, \cdot) = X^u(t_3,t_2, \cdot)\circ X^u(t_2,t_1, \cdot) \quad  \forall \, t_1, t_2, t_3 \in [0,T] \text{ and } x \in \T^d. 
\end{equation}
Clearly $X^u(t,x)$ as defined in \eqref{eq:flow-map} is just $X^u(t,x)=X^u(t,0,x)$, i.e. the flow with starting time $s =0$.

\subsection{Some preliminary observations}

We begin by some simple observations concerning extension of integral curves up to a possible singular time of the corresponding vector field.

\begin{lemma}\label{lemma:extension_flows}
    Let $v\in C^{\infty}([0,T)\times \T^d)\cap L^{\infty}([0,T]\times \T^d)$ and consider the associated (smooth) flow map with starting time
    \begin{equation*}
        X^v : [0,T)  \times [0,T) \times \T^d \to \T^d.
    \end{equation*}
    \begin{enumerate}[label=(\roman*)]
        \item \label{pt:extension} $X^v$ can be extended by continuity to a map
        \begin{equation*}
            X^v : [0,T] \times [0,T) \times \T^d \to \T^d 
        \end{equation*}
        \item \label{pt:convergence} $X^v(t, \cdot, \cdot) \to X(T, \cdot, \cdot)$ uniformly (in the variables $(s,x)$) as $t \nearrow T$. 
    
        \item \label{pt:semigroup} The semigroup property \eqref{eq:semigroup} holds also at time $t_3 = T$, i.e.
        \begin{equation}
        \label{eq:semigroup-extension}
            X^u(T,t_1, x) = X^u(T,t_2, x)\circ X^u(t_2,t_1, x), \quad  \forall  \,  t_1, t_2 \in [0,T) \text{ and } x \in \T^d. 
        \end{equation}        
    \end{enumerate}
\end{lemma}

\begin{remark}
    We will still call $X^v : [0,T] \times \T^d$ \emph{flow map with starting time} associated to $v$. However, $X^v(T,s, \cdot) : \T^d \to \T^d$ is not anymore, in general, a diffeomorphism.   
\end{remark}

\begin{proof}
We observe first of all that the family $\{X^v(t, \cdot, \cdot)\}_{t \in [0,T)}$ as a family of functions of the variables $(s,x) \in [0,T) \times \T^d$ is a Cauchy sequence (as $t \to T$) w.r.t. the uniform topology: indeed
\begin{equation*}
    |X^v(t_2, s,x) - X^v(t_1,s,x) | \leq \|v\|_{L^\infty([0,T] \times \T^d)} |t_2 - t_1|
\end{equation*}
and thus also
\begin{equation*}
    \sup_{(s,x) \in [0,T) \times \T^d}|X^v(t_2, s,x) - X^v(t_1,s,x) | \leq \|v\|_{L^\infty([0,T] \times \T^d)} |t_2 - t_1|.
\end{equation*}
Let us denote by $X^v(T, \cdot, \cdot) : [0,s) \times \T^d \to \T^d$ the uniform limit of the sequence $\{X^v(t,\cdot,\cdot)\}_{t \in [0,T)}$ as $t \to T$.
Clearly $X^v : [0,T] \times [0,T) \times \T^d \to \T^d$ (as a map of three variables) is continuous. This concludes the proof of Points 
\ref{pt:extension} and \ref{pt:convergence}.
Point \ref{pt:semigroup} follows from the semigroup property \eqref{eq:semigroup}, which holds for all $t_3 \in [0,T)$, passing to the limit $t_3 \nearrow T$ and using Point \ref{pt:convergence}.  
\end{proof}
\begin{lemma}
\label{lemma:composition_up_to_time_1}
    Let $u\in C^{\infty}([0,T]\times T^d)$ and $v\in C^{\infty}([0,T)\times \T^d)\cap L^{\infty}([0,T]\times \T^d)$.
    Consider the vector field $w\in C^{\infty}([0,T)\times \T^d)\cap L^{\infty}([0,T]\times \T^d)$ defined as in \eqref{eq:push_forward} by
    \begin{equation*}
         w(t,x):=u(t,x)+\left(\nabla \Phi^u(t,x)\right)^{-1}v(t,\Phi^u(t,x)).
    \end{equation*}
    Then the flow map of $w$, $X^w:[0,T]\times \T^d\to \T^d$
    \begin{equation}\label{eq:extensions_flows}
        X^w(t)=X^u(t)\circ X^v(t) \qquad \forall t\in[0,T].
    \end{equation}
\end{lemma}

\begin{remark}
    The flow map $X^u$ is well defined for all $t \in [0,T]$ because $u$ is smooth up to time $t=T$. The maps $X^v, X^w$ are well defined also at time $t=T$
    for all $t\in [0,T]$, by the previous lemma. 
\end{remark}
\begin{proof}
    For any $t<T$ the flow map of $v$ satisfies identity \eqref{eq:extensions_flows} thanks to Lemma \ref{lemma:composition_lemma}.
    By Lemma \ref{lemma:extension_flows}, we can continuously extend up to time $t=T$ both $X^v(t)$ and $X^w(t)$ by continuity. 
    Finally, formula \eqref{eq:extensions_flows} holds also at time $t=T$ because
    \begin{equation*}
        X^w(T, x) = \lim_{t\nearrow T} X^w(t,x)=X^u(T,\lim_{t\nearrow T}X^v(t,x))=X^u(T, X^v(T,x)) \quad \forall x\in \T^d
    \end{equation*}
    by smoothness of $X^u$ on $[0,T]\times \T^d$.
 \end{proof}


\subsection{Strategy for non uniqueness}

We briefly describe here the strategy developed by Kumar in \cite{kumar2023} to produce vector fields for which uniqueness of characteristics fails on a set of full measure. 
Let 
$$\bar v\in C^\infty([0,T) \times \T^d) \cap  C([0,T];W^{1,p}\cap C(\T^d;\R^d))$$
be a given incompressible vector field.
Observe that, by Lemma \ref{lemma:extension_flows}, the (smooth) flow map $X^{\bar v}$ associated to $\bar v$ can be extended by continuity up to the singular time $t=T$. Denote by $v$ the time-reversed vector field:
$$v(t,x):=-\overline v(T-t,x).$$ 
Observe that 
$$v \in C^\infty((0,T] \times \T^d) \cap  C([0,T];W^{1,p}\cap C(\T^d;\R^d)).$$
Recall again the definition of the set $A_{v}$ in \eqref{eq:non-uniq-ode}.

\begin{lemma}
\label{lemma:nonuniq_argument} 
    Assume there exists a set $C\subset \T^d$ of null (Lebesgue) measure such that
    \begin{equation}
    \label{eq:strategy-1}
        X^{\bar v}(T,C) = \T^d. 
    \end{equation}
    Then the set $A_{v}$ has full measure.
\end{lemma}


\begin{remark}
    Recall that $X^{\bar v}$ is well defined by Lemma \ref{lemma:extension_flows}.
\end{remark}

\noindent We present a proof of the lemma for the sake of completeness. Just for this proof, it will be convenient to denote by $\mathcal{L}^d$ the $d-$dimensional Lebesgue measure.
\begin{proof}
    We prove that
    \begin{equation*}
        \mathcal{L}^d \left( \left\{ x \in \T^d \, : \, \exists ! \ \gamma  \text{ integral curve of $v$ with } \gamma(0) = x\right\} \right) = 0.
    \end{equation*}
    Equality \eqref{eq:strategy-1} means that for every $x \in \T^d$ there is an integral curve $\bar \gamma_x$ of $\bar v$ such that $\bar \gamma_x(0) \in C$ and $\bar \gamma_x(T) = x$. Observe then that the curve
    \begin{equation*}
        \gamma_x (t) := \bar \gamma_x(T-t)
    \end{equation*}
    is an integral curve of $v$, with $\gamma_x(0) = x$. Now, since $v$ has Sobolev regularity and bounded divergence (actually zero divergence), DiPerna-Lions theory implies that $v$ admits a (unique) regular Lagrangian flow $X$, with compressibility constant $L>0$, i.e. $X(t)_\sharp \mathcal L^d \leq L \mathcal L^d$.
    Observe that 
    \begin{equation*}
    \begin{aligned}
        A_v^c 
        & = \big\{ x \, : \,  \exists ! \  \text{ integral curve of $v$ with } \gamma(0) = x\big\} \\
        & \subseteq 
        \left\{x \ \, : \ X(\cdot, x) = \gamma_x \right\}    
        \cup 
        \left\{ x  \, : \, X(\cdot, x) \text{ is a not an integral curve of } v  \right\}\\
        & \subseteq 
        X(T)^{-1}(C) \cup 
        \left\{ x  \, : \, X(\cdot, x) \text{ is a not an integral curve of $v$} 
        \right\}.
    \end{aligned}
    \end{equation*}
The first inclusion follows from the fact that $\gamma_x$ is an integral curve of $v$ starting in $x$, for all $x$. The second inclusion follows from the fact that $\gamma_x(T) \in C$ for all $x$. Now simply notice that, by definition of regular Lagrangian flow,
\begin{equation*}
    \mathcal{L}^d (X(T)^{-1}(C)) = X(T)_\sharp (\mathcal{L}^d) (C) \leq L \mathcal{L}^d(C) = 0,
\end{equation*}
and, again by definition of regular Lagrangian flow, 
\begin{equation*}
    \mathcal{L}^d \left( \left\{ x  \, : \, X(\cdot, x) \text{ is a not an integral curve of $v$} 
        \right\} \right) = 0. 
\end{equation*}
Hence $\mathcal{L}^d(A_v^c)=0$ and thus $A_v$ has full measure. 
\end{proof}

\subsection{Kumar's construction (adapted to get $W^{1,p}$ smallness)}

The following theorem, but without the smallness assumption \eqref{eq:bb_properties}, was proven in \cite{kumar2023}. We present below a sketch of the proof, showing how Kumar's construction can be easily adapted to achieve also \eqref{eq:bb_properties}.

\begin{theorem}[Kumar, adapted to get smallness]
\label{prop:nonuniq_bb}
	Let $p<d$ and $\alpha\in(0,1)$ be fixed. There is $T_0>0$ such that for every $\varepsilon>0$ there exist a divergence free vector field 
    \begin{equation*}
        \bar v\in C^{\infty}([0,T_0)\times \T^d;\R^d)\cap C([0,T_0];W^{1,p}(\T^d;\R^d))\cap C^{\alpha}([0,T_0]\times \T^d;\R^d)
    \end{equation*}
    such that 
	\begin{equation}\label{eq:bb_properties}
		\norm{\bar v}_{C_tW^{1,p}_x\cap C^{\alpha}_{t,x}}<\varepsilon
	\end{equation}
     and having the following property.
     There exists a set 
     $C\subset \T^d$, with $|C|=0$, such that:
     \begin{equation}\label{eq:flow_cantor}
         X^{\bar v}(T_0,C)=\T^d,
     \end{equation}
      where $X^{\bar v}$ is the unique flow of $v$ defined up to time $t=T_0$.
\end{theorem}
\begin{remark}
    Recall that $X^{\bar v}$ is well defined up to time $T_0$ by Lemma \ref{lemma:extension_flows}.
\end{remark}

\begin{proof}[Sketch of the proof]

For given $p<d$, $\alpha \in (0,1)$, in \cite{kumar2023} the author constructs an incompressible vector field $$\mathrm{v} \in C^{\infty}([0,T_0)\times \T^d;\R^d)\cap C([0,T_0];W^{1,p}(\T^d;\R^d))\cap C^{\alpha}([0,T_0]\times \T^d;\R^d),$$ for some $T_0>0$, that can be written as
\begin{equation*}
    \rm v(t,x)=\sum_{i=1}^{\infty}\rm v_{i}(t,x),
\end{equation*}
where the family of vector fields $\{\rm v_{i}\}_{i\in\N}$ has the following properties:
\begin{enumerate}
    \item $\mathrm{ v_{i}}\in C^{\infty}([0,T_0]\times \T^d;\R^d)$ for any $i\in\N$;
    \item $\rm v_{i}$ is divergence free for any $i\in\N$;
    \item there exists a sequence of times $\{\tau_i\}_{i\in\N}$, with $\tau_j < \tau_{j+1}$ for all $j\in\N$ and $\tau_j\to T_0$, such that supp$_t \,\rm v_{i}\,\subseteq (\tau_i,\tau_{i+1})$; 
    \item the sequence $\{\rm v_{i}\}_{i\in\N}$ is summable in $C([0,T_0];W^{1,p}(\T^d))\cap C^{\alpha}([0,T_0]\times \T^d)$.
\end{enumerate}

The crucial property of $\rm v$ is that  there exists a zero measure Cantor set, which we denote by $\mathcal{K}\subset \T^d$, $\mathcal{L}^d(\mathcal{K})=0$, such that 
\begin{equation}
\label{eq:cantor}
    X^{\rm v}(T_0,\mathcal{K})=\T^d.
\end{equation}
Notice that by Lemma \ref{lemma:extension_flows}  the flow $X^{\rm v}$ is well defined up to time $t=T_0$.
For a fixed $n$, we define the (incompressible) vector field
\begin{equation*}
    v_{n}(t,x):=\sum_{i=n}^{\infty}\mathrm{v}_{i}(t,x).
\end{equation*}
and the corresponding set 
\begin{equation}\label{eq:cantor_n}
    \mathcal{K}^n:=X^{\rm v}(\tau_n,0,\mathcal{K}).
\end{equation}
Observe that $\rm v$ is smooth on $[0,\tau_n] \subseteq [0,T_0)$ and thus 
$\mathcal{L}^d(\mathcal{K}^n)=0$.
From definition \eqref{eq:cantor_n} and from the semigroup property
of the flow up to time $t=T_0$, which holds thanks to point \eqref{pt:semigroup} of Lemma \ref{lemma:extension_flows}, we deduce that 
\begin{equation}
\label{eq:cantor-vn}
\begin{aligned}
     &X^{v_n}(T_0,\mathcal{K}^n)&&=X^{v_n}(T_0,0,\mathcal{K}^n)\\
     & \textit{\small(semigroup property of $X^{v_n}$)} &&=X^{v_n}(T_0,\tau_n,X^{v_n}(\tau_n,0,\mathcal{K}^n)\\
     & \textit{\small ($v_n=0$ for $t\le \tau_n$)} && = X^{v_n}(T_0,\tau_n,\mathcal{K}^n) \\
     & \textit{\small ($\mathtt v = v_n$ on $[\tau_n, T_0]$)}
     && = X^{\rm v} (T_0,\tau_n,\mathcal{K}^n)\\
     &  \textit{\small(by definition of $\mathcal{C}^n$)} 
     &&=X^{\rm v} (T_0,\tau_n,X^{\mathtt v}(\tau_n,0,\mathcal{K}))\\
     &\textit{\small (semigroup property of $X^{ \rm v}$)} &&=X^{\rm v}(T_0,0,\mathcal{K}) =X^{\rm v}(T_0,\mathcal{K})\\
     & \textit{\small (by \eqref{eq:cantor})} && =\T^d.
\end{aligned}
\end{equation}
Notice now that, by summability of the sequence $\{\rm v_{i}\}_{i\in\N}$ in $C([0,T_0];W^{1,p}(\T^d))\cap C^{\alpha}([0,T_0]\times \T^d)$, the vector field $\bar v$ defined as $v_n$, for $n$ sufficiently large, satisfies \eqref{eq:bb_properties} and, by \eqref{eq:cantor-vn}, it satisfies also \eqref{eq:flow_cantor}, with $C := \mathcal{K}^n$.  
\end{proof}

\subsection{Proof of Point (ii) of Theorem \ref{thm:density}}
As observed in the proof of Point (i), see Section \ref{ss:proof-i}, we can assume w.l.o.g. that $T=T_0$, where $T_0$ is the time given by Theorem \ref{prop:nonuniq_bb}. The general case follows from the rescaling $t \mapsto tT/T_0$, $u \mapsto uT/T_0$. 

By density of smooth maps in $\mathcal{Z}$ it is enough to show that for every incompressible vector field $u \in C^\infty([0,T_0] \times \T^d)$ and for every $\varepsilon>0$ there is another incompressible vector field $$w\in  C([0,T_0];W^{1,p}(\T^d;\R^d))\cap C^{\alpha}([0,T]\times \T^d;\R^d)$$ such that 
	\begin{equation*}
		\norm{u-w}_{C_tW^{1,p}_x\cap C_{t,x}}<\varepsilon
	\end{equation*}
	and 
    \begin{equation}
    \label{eq:kumar3}
     \text{$A_w$ has full measure}
    \end{equation}

\noindent (recall the definition of $A_w$ in \eqref{eq:non-uniq-ode}). 

Let thus $u \in C^\infty$ and $\varepsilon>0$ be given. Let $\bar u (t,x):=-u(T_0-t,x)$ be the time reversal  of $u$. By Theorem \ref{prop:nonuniq_bb}, there is  
\begin{equation}
\label{eq:reg-v}
        \bar v\in C^{\infty}([0,T_0)\times \T^d;\R^d)\cap C([0,T_0];W^{1,p}(\T^d;\R^d))\cap C^{\alpha}([0,T_0]\times \T^d;\R^d)
\end{equation}
and a set $C \subseteq \T^d$, $\mathcal{L}^d(C) = 0$ such that $$X^{\bar v} (T_0, C) = \T^d$$ and
\begin{equation}
\label{eq:size-v}
    \norm{\left(\nabla \Phi^{\bar u}\right)^{-1}\bar v(t,\Phi^{\bar u})}_{C_t(W^{1,p}\cap C)_x} < \varepsilon,
\end{equation}
where $\Phi^{\bar u}$ is the inverse flow map associated to $\bar u$. Set
\begin{equation*}
	\bar w(t,x):=\bar u(t,x) + \left(\nabla \Phi^{\bar u}(t,x)\right)^{-1}\bar v(t,\Phi^{\bar u}(t,x)).
\end{equation*}
Observe that by \eqref{eq:size-v}, 
\begin{equation}
    \label{eq:size-u-minus-w}
\norm{\bar w - \bar u }_{C_t(W^{1,p}\cap C)_x} < \varepsilon. 
\end{equation}

Since $\bar u$ and $\Phi^{\bar u }$ are smooth, the vector field $\bar w$ inherits the same regularity as $\bar v$ in \eqref{eq:reg-v}. 
Moreover, since $\bar v$ and $\bar u$ are incompressible and $\bar u$ is also smooth, by Remark \ref{rmk:div-u-equal-div-w} the vector field $\bar w$ has zero distributional divergence.
We can thus apply Lemma  \ref{lemma:composition_up_to_time_1} to get
\begin{equation*}
	X^{\bar w}(T_0,C)=X^{\bar u}(T_0,X^{\bar v}(T_0,C))=X^{\bar u}(T_0,\T^d)=\T^d,
\end{equation*}
where, in the last equality, we  used that  $X^{\bar u}$ is a smooth diffeomorphism of $\T^d$. Therefore $\bar w$  satisfies the assumptions of Lemma \ref{lemma:nonuniq_argument} and thus its time reversal $w(t) := \bar w(T_0-t)$ has the property \eqref{eq:kumar3}. Moreover, by \eqref{eq:size-u-minus-w},
\begin{equation*}
\norm{w - u }_{C_t(W^{1,p}\cap C)_x} = \norm{\bar w - \bar u }_{C_t(W^{1,p}\cap C)_x} < \varepsilon,
\end{equation*}
thus concluding the proof of Theorem \ref{thm:density}. \qed

\begin{remark}
\label{rmk:holder-not-work}
    By construction, the perturbation $v$ coming from Proposition \ref{prop:nonuniq_bb} is arbitrarily small not only in the uniform norm but also in $C^{\alpha}$. 
    In other words, from the previous proof we could deduce that any smooth function is an accumulation point in $C^{\alpha}$ of the subset of vector fields having non-uniqueness of trajectories in almost every point. 
    On the other hand, $C^{\infty}$ is not dense in $C^{\alpha}$ with its natural topology, so that we can not state a density result also in this setting.
\end{remark}

 \bibliographystyle{siam}
 \bibliography{bibliography}

\end{document}